\documentclass{elsarticle}
\usepackage{ amstext, amsmath, amssymb, graphicx, array, epsfig, psfrag}
\usepackage{amstext, amsmath, amssymb, graphicx}
\usepackage{todonotes}

\def\IR{\mathbb R}

\newcommand{\bfn}{\boldsymbol n}
\newcommand{\bfI}{\boldsymbol I}
\newcommand{\bfP}{\boldsymbol P}
\newcommand{\bfx}{\boldsymbol x}
\newcommand{\bfy}{\boldsymbol y}
\newcommand{\bfz}{\boldsymbol z}
\newcommand{\bfa}{\boldsymbol a}

\newcommand{\bfb}{\boldsymbol b}
\newcommand{\bfp}{\boldsymbol p}

\newcommand{\bfB}{\boldsymbol B}

\newcommand{\bfkappa}{\boldsymbol \kappa}

\newcommand{\mcV}{{V}}
\newcommand{\mcK}{\mathcal{K}}

\newcommand{\mcN}{\mathcal{N}}
\newcommand{\mcF}{\mathcal{F}}
\newcommand{\mcA}{\mathcal{A}}

\newcommand{\mcX}{\mathcal{X}}

\newcommand{\tn}{|\mspace{-1mu}|\mspace{-1mu}|}

\numberwithin{equation}{section}
\newtheorem{lem}{Lemma}[section]

\newtheorem{thm}{Theorem}[section]

\newenvironment{proof}{\noindent \newline {\bf Proof.}}
{\hfill \mbox{\fbox{} } \newline}
\newcommand{\nablas}{\nabla_\Sigma}
\newcommand{\nablash}{\nabla_{\Sigma_h}}

\newcommand{\bfPs}{{\boldsymbol P}_\Sigma}
\newcommand{\bfPsh}{{\boldsymbol P}_{\Sigma_h}}

\newcommand{\ds}{d \sigma}
\newcommand{\dsh}{d \sigma_h}

\begin{document}
\title{A Stabilized Cut Finite Element Method for Partial Differential Equations on Surfaces: The Laplace-Beltrami Operator}
\author[UCL]{Erik Burman}
\author[HJ]{Peter Hansbo}
%\ead{peter.hansbo@jth.hj.se}
\author[UMU]{Mats G.\ Larson}
\address[UCL]{Department of Mathematics, University College London, London, UK--WC1E 6BT, United Kingdom} 
\address[HJ]{Department of Mechanical Engineering, J\"onk\"oping University,
SE-55111 J\"onk\"oping, Sweden} 
\address[UMU]{Department of Mathematics and Mathematical Statistics, Ume{\aa} University, 
SE-901 87 Ume{\aa}, Sweden} 
%\numberwithin{equation}{section} \maketitle
%%%%%%%%%%%%%%%%%%%%%%%%%%%%%%%%%%%%%%%%%%%%%%%%%%%%%%%%%%%%%%%%%%%%%%%
\begin{abstract}
We consider solving the Laplace-Beltrami problem on a smooth two 
dimensional surface embedded into a three dimensional space meshed 
with tetrahedra. The mesh does not respect the surface and thus the 
surface cuts through the elements. We consider a Galerkin method 
based on using the restrictions of continuous piecewise linears 
defined on the tetrahedra to the surface as trial and test functions.

The resulting discrete method may be severely ill-conditioned, and 
the main purpose of this paper is to suggest a remedy for this 
problem based on adding a consistent stabilization term to the 
original bilinear form. We show optimal estimates for the condition 
number of the stabilized method independent of the location of the 
surface. We also prove optimal a priori error estimates for the 
stabilized method.
\end{abstract}
\maketitle

\section{Introduction}

We consider solving the Laplace-Beltrami problem on a smooth two 
dimensional surface embedded into a three dimensional space partitioned 
into a mesh consisting of shape regular tetrahedra. The mesh does not 
respect the surface and thus the surface cuts through the elements. Following 
Olshanskii, Reusken, and Grande \cite{OlReGr09} we construct a Galerkin method 
by using the restrictions of continuous piecewise linears defined on the tetrahedra 
to the surface.

The resulting discrete method may be severely ill-conditioned and 
the main purpose of this paper is to suggest a remedy for this problem 
based on adding a consistent stabilization term to the original bilinear form. The stabilization 
term we consider here controls jumps in the normal gradient on the faces 
of the tetrahedra and provides a certain control of the derivative of the discrete functions 
in the direction normal to the surface. Similar terms have recently been used for 
stabilization of cut finite element methods for fictitious domain methods \cite{BuHa10}, \cite{BuHa12}, \cite{HaLaZa14}, and \cite{MaLaLoRo12}. Note that 
none of these references involve any partial differential equations on surfaces only regular boundary and interface conditions.

In principle, it is possible, in this situation, to deal with the ill conditioning problem in the linear algebra using a scaling, see \cite{OlRe10}. Starting from a stable method has clear advantages in more complex applications that may need stabilization anyway, such as problems with hyperbolic character or coupled bulk-surface problems. It is also not 
clear that matrix based algebraic scaling procedures is possible in all 
situations and thus alternative approaches must be investigated.

Using the additional stability we first prove an optimal estimate 
for the condition number, independent of the location of the surface, 
in terms of the mesh size of the underlying tetrahedra. The key step 
in the proof is certain discrete Poincar\'e estimates that are also 
of general interest. Then we prove a priori error estimates in the 
energy and $L^2$ norms.  

In a companion paper, we will consider the  more challenging 
problems of the surface Helmholtz equation and show error estimates for a stabilized method under a suitable condition on the product of the mesh size and the wave number.

Finally, we refer to \cite{DeDzElHe10}, \citep{Dz88}, \cite{DzEl08}, and \cite{DzEl13} for general background on 
finite element methods for partial differential equations on surfaces.

The outline of the reminder of this paper is as follows: In Section 
2 we formulate the model problem and the finite element method, in 
Section 3 we summarize some preliminary results involving lifting of 
functions from the discrete surface to the continuous surface,
in Section 4 we prove an optimal bound on the condition number of 
the stabilized method, in Section 5 we prove a priori error estimates 
in the energy and $L^2$ norms, and finally in Section 6 we present 
numerical investigations confirming our theoretical results.

\section{Model Problem and Finite Element Method}

\subsection{The Continuous Surface}

Let $\Sigma$ be a smooth $d-1$-dimensional closed surface embedded in ${{\IR}}^d$, $d=2$ or $3$, with signed distance function $\rho$ such 
that the exterior surface unit normal is given by $\bfn=\nabla \rho$. 
Let $\bfp(\bfx)$ be the nearest point projection mapping onto $\Sigma$, i.e., $\bfp(\bfx)$ is the point on $\Sigma$ that minimizes the Euclidian 
distance to $\bfx$. For $\delta>0$ let $U_\delta(\Sigma)$ be the tubular neighborhood $U_\delta(\Sigma) = \{ \bfx \in \IR \,:\, |\rho(\bfx)| < \delta \}$ of $\Sigma$. Then $\bfp(\bfx) = \bfx - \rho(\bfx) \bfn(\bfp(\bfx))$ and there 
is a $\delta_0>0$ such that for each $\bfx \in U_{\delta_0}(\Sigma)$ 
there is a unique $\bfp(\bfx) \in \Sigma$.  Using $\bfp$ we may extend 
any function $v$ defined on $\Sigma$ to $U_{\delta_0}(\Sigma)$ by defining 
\begin{equation}\label{extension}
v^e(\bfx) = v \circ \bfp (\bfx), \quad \bfx \in U_{\delta_0}(\Sigma)
\end{equation}

\subsection{The Continuous Problem}

We consider the following problem: find $u: \Sigma \rightarrow {{\IR}}$ 
such that
\begin{align}\label{eq:LB}
-\Delta_\Sigma u = f \quad \text{on $\Sigma$}
\end{align}
where $f$ is a given function such that $\int_\Sigma f = 0$. Here $\Delta_\Sigma$ is the Laplace-Beltrami operator defined by
\begin{equation}
\Delta_\Sigma = \nabla_\Sigma \cdot \nabla_\Sigma
\end{equation}
where $\nabla_\Sigma$ is the tangent gradient
\begin{equation}
\nabla_\Sigma = \bfPs \nabla
\end{equation}
with $\bfPs = \bfPs(\bfx)$ the projection of $\IR^3$ onto the tangent 
plane of $\Sigma$ at $\bfx\in\Sigma$, defined by
\begin{equation}
\bfPs = \bfI -\bfn \otimes \bfn
\end{equation}
where $\bfI$ is the identity matrix, and $\nabla$ the ${{\IR}}^3$ 
gradient. 

Let $(v,w)_\omega = \int_\omega v w$ and $\|v\|_\omega = (v,v)_\omega$ 
be the $L^2(\omega)$ inner product and norm on the set $\omega$ 
equipped with the appropriate measure. Let $H^m(\Sigma), m = 0, 1, 2$ 
be the Sobolev spaces on $\Sigma$ with norm 
\begin{equation}
\| w \|^2_{m,\Sigma} = \sum_{s=0}^m \| (D_\Sigma^P)^s w \|_\Sigma^2
\end{equation} 
where
\begin{gather}
(D_\Sigma^P)^0 w=w,
\quad (D_\Sigma^P)^1 w=\nabla_\Sigma w 
\\
(D_\Sigma^P)^2 w= \bfPs ((\nabla_\Sigma w)\otimes \nabla_\Sigma) 
=\bfPs (\nabla \otimes \nabla w) \bfPs  
\end{gather}
and the $L^2$ norm for a matrix is based on the pointwise 
Frobenius norm. We then have the weak problem: find 
$u \in H^1(\Sigma)/\IR$ such that
\begin{equation}
a(u,v) = l(v) \quad \forall v \in H^1(\Sigma)/\IR
\end{equation}
where
\begin{equation}
a(u,v) = (\nabla_\Sigma u, \nabla_\Sigma v)_\Sigma, \quad l(v) = (f,v)_\Sigma
\end{equation}
It follows from the Lax-Milgram lemma that the weak problem has 
a unique solution for $f\in H^{-1}(\Sigma)$ such that 
$\int_\Sigma f = 0$. For smooth surfaces we also have the elliptic regularity estimate
\begin{equation}
\| u \|_{2,\Sigma} \lesssim \|f\|_\Sigma
\end{equation}
Here and below $\lesssim$ denotes less or equal up to a positive 
constant.

\subsection{Approximation of the Surface}

Let $\mcK_{h,0}$ be a quasi uniform partition into shape regular 
tetrahedra for $d=3$ and triangles for $d=2$ with mesh parameter 
$h$ of a polygonal domain $\Omega_0$ in $\IR^d$ completely 
containing $U_{\delta_0}(\Sigma)$. 
Let $\mcV_{h,0}$ be the space of continuous piecewise linear 
polynomials defined on $\mcK_{h,0}$. Let $\rho_h\in \mathcal{V}_{h,0}$ 
be an approximation of the distance function $\rho$ and let 
$\Sigma_h$ be the zero levelset
\begin{equation}
\Sigma_h = \{ \bfx \in \Omega_0 : \rho_h(\bfx) = 0 \}
\end{equation}
Then $\Sigma_h$ is piecewise linear and we define the exterior 
normal $\bfn_h$ to be the exact exterior unit normal to 
$\Sigma_h$. We consider a family of such surfaces 
$\{\Sigma_h : 0 < h \leq h_0\}$ such that (a) 
$\Sigma_h \subset U_{\delta_0}(\Sigma)$, 
(b) the closest point mapping $\bfp:\Sigma_h \rightarrow \Sigma$, 
is a bijection, and (c) the following estimates 
hold
\begin{equation}\label{eq:rhobounds}
\| \rho \|_{L^\infty(\Sigma_h)} \lesssim h^2, 
\qquad \| \bfn^e - \bfn_h \|_{L^\infty(\Sigma_h)} \lesssim  h
\end{equation}
for $0<h \leq h_0$. These properties are, for instance, satisfied 
if $\rho_h$ is the Lagrange interpolant of $\rho$ and $h_0$ is small 
enough.

%In practice we are typically not able to compute on the exact 
%surface $\Sigma$ instead we have to consider an approximate 
%surface $\Sigma_h$. Depending on how the surface is described the construction of the approximate surface can be done in different ways. 
%Here we consider, in particular, a simple situation where $\Sigma$ 
%is the zero levelset to the distance function $\rho_h$ and $\Sigma_h$ 
%is defined as the zero levelset to a piecewise linear approximate 
%distance function $\rho_h$. In this case the approximate 
%surface is piecewise linear surface since it is a levelset to a 
%piecewise linear function. 
%
%We consider a family of discrete surfaces 

\subsection{The Finite Element Method}

Let
\begin{equation}
\mcK_h = \{ K \in \mcK_{h,0} : \overline{K} \cap \Sigma_h \neq \phi \}, \quad \Omega_h = \cup_{K \in \mcK_h} K
\end{equation}
and 
\begin{equation}
\mcV_h = \{ v \in V_{h,0}|_{\Omega_h} \,:\, \int_{\Sigma_h} v = 0 \}
\end{equation}
be the continuous piecewise linear functions defined on $\mcK_h$ with 
average zero. The finite element method 
on $\Sigma_h$ takes the form: find $u_h \in \mcV_h$ such that
\begin{equation}\label{eq:fem}
A_h(u_h,v) = l_h(v) \quad \forall v \in \mcV_h
\end{equation}
Here the bilinear form $A_h(\cdot, \cdot)$ is defined by
\begin{equation}
A_h(v,w) = a_h(v,w) + j_h(v,w) \quad \forall v,w \in \mcV_h
\end{equation}
with
\begin{equation}
a_h(v,w) = (\nabla_{\Sigma_h} v, \nabla_{\Sigma_h} w)_{\Sigma_h}
\end{equation}
and
\begin{equation}
j_h(v,w) = \sum_{F \in \mcF_h} (\tau_0[\bfn_F \cdot \nabla v],[\bfn_F \cdot \nabla w])_F
\end{equation}
where $\mcF_h$ denotes the set of internal interfaces in $\mcK_h$, 
$[\bfn_F \cdot \nabla v] = (\bfn_F \cdot \nabla v)^+ 
- (\bfn_F \cdot \nabla v)^-$ with 
$w(\bfx)^\pm = \lim_{t\rightarrow 0^+} w(\bfx \mp t \bfn_F)$, is 
the jump in the normal gradient across the face $F$,  $\bfn_F$ 
denotes a fixed unit normal to the face $F\in \mcF_h$, and $\tau_0$ is a constant of $O(1)$. The tangent 
gradients are defined using the normal to the discrete surface
\begin{equation}
\nabla_{\Sigma_h} v = \bfPsh \nabla v = ( \bfI - \bfn_h \otimes \bfn_h ) \nabla v 
\end{equation}
and the right hand side is given by
\begin{equation}
l_h(v) = ( f^e, v )_{\Sigma_h}
\end{equation}
%where $f_h$ is a function such that $\int_{\Sigma_h} f_h = 0$ and
%\begin{equation}\label{assumpf}
%\| f^e - f_h \|_{L^\infty(\Sigma_h)} \lesssim h^2
%\end{equation}
Introducing the mesh dependent norm
\begin{equation}
\tn v \tn^2_{h}  = \tn v \tn^2_{\Sigma_h} + \tn v \tn^2_{\mcF_h}
\end{equation}
where
\begin{equation}
\tn v \tn^2_{\Sigma_h} = \| \nabla_{\Sigma_h} v \|_{\Sigma_h}^2,
\qquad 
\tn v \tn^2_{\mcF_h} = j_h(v,v)
\end{equation}
we note that $\tn v \tn^2_{h}$ is indeed a norm on $\mcV_h$ 
(for fixed $\Sigma_h$) since if $\tn v \tn_h = 0$ we have $v=0$. 
To see this we first note that if $\tn v \tn_{\mcF_h}=0$ then 
$v$ must be a linear polynomial $v(\bfx)=\bfa\cdot (\bfx-\bfx_{\Sigma_h}) 
+ b$ with $\bfa \in \IR^d$, $b\in \IR$, and $\bfx_{\Sigma_h}\in \IR^d$ 
the center of gravity of $\Sigma_h$, here $b=0$ since 
$0=\int_{\Sigma_h} v = b$, secondly if $0=\tn v \tn_{\mcF_h} = \|\bfPsh \bfa \|_{\Sigma_h}$ and thus $\bfa=\boldsymbol{0}$ since $\Sigma_h$ 
is a closed surface and thus $\bfa$ cannot be normal to $\Sigma_h$ everywhere. Thus it follows from the Lax-Milgram lemma that there 
exists a unique solution to (\ref{eq:fem}).

\section{Preliminary Results}
In this section we collect some essentially standard results, 
see \cite{De09}, \cite{DeDz07}, and \cite{Dz88},
related to lifting of functions from the discrete surface 
to the exact surface.

\subsection{Lifting to the Exact Surface}
For each $K \in \mcK_h$ we let $\rho_{h,K}$ be the distance function to the 
hyperplane, with normal $\bfn_{h,K}$, that contains $K \cap \Sigma_h$ and 
$\bfp_{h,K}$ be the associated nearest point projection 
$\bfp_{h,K}(\bfx) = \bfx - \rho_{h,K}(\bfx) \bfn_{h,K}$ onto the hyperplane. 
Then we define the mapping
\begin{equation}
G_K: K \ni \bfx \mapsto \bfp \circ \bfp_{h,K}(\bfx) + \bfn(\bfx) \rho_{h,K}(\bfx) \in K^l 
\end{equation}
where $K^l = G_K(K)$ and we defined $\bfn(\bfx) = \nabla \rho(\bfx) = \bfn^e(\bfx)$ for $\bfx \in U_{\delta_0}(\Sigma)$. We note that $G_K$ 
is an invertible mapping, $G_K(K \cap \Sigma_h) = G_K(K) \cap \Sigma$, 
and $\{K^l \cap \Sigma : K^l = G_K(K), K\in \mcK_h \}$ is a partition of $\Sigma$. The derivative $DG_K$ of $G_K$ is given by
\begin{equation}
DG_K = (\bfPs - \rho \bfkappa ) \bfPsh + \bfn \otimes \bfn_h 
+ \rho_{h,K} \bfkappa ({\bfPs} - \rho \bfkappa )
\end{equation}
where we used the identity $\bfp \otimes \nabla = \bfPs - \rho \bfkappa$. Here $\bfkappa = \nabla\otimes\nabla \rho$ and we have the identity 
\begin{equation}
\bfkappa = \sum_{i=1}^2 \kappa_i^e (1 + \rho \kappa_i^e)^{-1} \bfa_i^e\otimes \bfa_i^e
\end{equation}
where $\kappa_i$ are the principal curvatures of $\Sigma$ with 
corresponding principal curvature vectors $\bfa_i$, see \cite{GiTr83} 
Lemma 14.7. Thus we note that $\bfkappa$ is tangential to $\Sigma$ and that for $\delta_0$ small enough we have the estimate 
$\|\bfkappa\|_{L^\infty(U_{\delta_0}(\Sigma))} \lesssim 1$. 

In particular, on $\Sigma_h$ we have $\rho_{h,K}(\bfx) = 0$ and thus we obtain the simplified expression
\begin{equation}
DG_K = (\bfI - \rho \bfkappa ) (\bfPs \bfPsh + \bfn \otimes \bfn_h )
\end{equation}
where we used the fact that $\bfkappa$ is a tangential tensor to 
$\Sigma$, i.e. $\bfkappa \bfPs = \bfkappa$. The mapping $DG_K$ maps 
the tangent and normal spaces of $\Sigma_h$ at $\bfx \in K \cap \Sigma_h$ onto the tangent and normal spaces of $\Sigma$ at $\bfp(\bfx)$.

We define the lift $v^l$ of a function $v \in H^1(K)$ to $H^1(K^l)$ by 
\begin{equation}
v^l \circ G_K = v
\end{equation}

\subsection{Tangent Gradients of Lifted Functions}

The tangent gradient of $v^l$ is given by 
\begin{equation}
\nablas v^l = \bfPs \nabla v^l = \bfPs DG_K^{-T} \nabla v = 
\bfPs (\bfI - \rho \bfkappa)^{-T} ({\bfPs} \bfPsh + \bfn \otimes \bfn_h )^{-T} \nabla v
\end{equation}
Using the fact that $\bfI - \rho {\bfkappa}$ and  ${\bfPs} \bfPsh + {\bfn} \otimes \bfn_h$ 
preserves the normal and tangent directions we finally obtain 
the identity
\begin{equation}
\nablas v^l = 
(\bfI - \rho {\bfkappa} )^{-T} ({\bfPs} \bfPsh + {\bfn} \otimes \bfn_h )^{-T} \bfPsh \nabla v
= \bfB^{-T} \nabla_{\Sigma_h} v
\end{equation}
where we introduced the notation 
\begin{equation}
\bfB = (\bfI - \rho \bfkappa) (\bfPs \bfPsh + \bfn \otimes \bfn_h )
\end{equation}
Here $\bfI - \rho {\bfkappa}$ is a symmetric matrix with eigenvalues 
%$\{1,1 - \rho {\kappa^e_1}(1+\rho\kappa^e_1)^{-1}, 
%1 - \rho {\kappa^e_2}(1+\rho\kappa^e_2)^{-1}\}$, 
$\{1,(1+\rho\kappa^e_1)^{-1}, (1+\rho\kappa^e_2)^{-1}\}$,
which are all strictly greater than zero on $U_{\delta_0}(\Sigma)$ 
for $\delta_0$ small enough,
%
%$\delta_0 \leq c_0 \max (\|\kappa_1\|_{L^\infty(\Sigma)}^{-1},\|\kappa_2\|%_{L^\infty(\Sigma)}^{-1})$ 
%with $0< c_0 < 1$
and
$\bfPs \bfPsh + {\bfn} \otimes \bfn_h$ is nonsymmetric with singular 
values $\{1,1, \bfn \cdot \bfn_h\}$, which are also strictly greater 
than zero. We will need the following estimates 
for $\bfB$:
\begin{equation}\label{Bestimates}
\| \bfB \|_{L^\infty(\Sigma_h)} \lesssim 1, \quad 
\| \bfB^{-1} \|_{L^\infty(\Sigma_h)} \lesssim 1, \quad
\| \bfI - \bfB \bfB^T \|_{L^\infty(\Sigma_h)} \lesssim  h^2
\end{equation}
The first estimate follows directly from the definition of $\bfB$, the second 
can be proved using the the fact that eigenvalues and singular values discussed 
above are all strictly greater than zero. To prove the third we proceed as follows
\begin{align}\nonumber
&\| \bfI - \bfB\bfB^T \|_{L^\infty(\Sigma_h)} 
\\
\qquad&= 
\| \bfI - (\bfI - \rho \bfkappa) ( \bfPs \bfPsh + \bfn \otimes \bfn_h )
\\ \nonumber
&\qquad \qquad \times ( \bfPsh \bfPs + \bfn_h \otimes \bfn ) (\bfI - \rho \bfkappa) \|_{L^\infty(\Sigma_h)} 
 \\
\qquad &= \| \bfI - \bfPs \bfPsh \bfPs - \bfn \otimes \bfn \|_{L^\infty(\Sigma_h)} + O(h^2)
\end{align}
where we collected the terms involving the distance function $\rho$ in 
the last term and used the assumption (\ref{eq:rhobounds}) that 
$\| \rho \|_{L^\infty(\Sigma_h)} \lesssim h^2$. Next for the remaining 
term we write $\bfI = \bfPs + \bfn \otimes \bfn$ and then we 
note that the identity 
\begin{equation}
\bfPs - \bfPs \bfPsh \bfPs = \bfPs (\bfPs - \bfPsh )(\bfPs - \bfPsh ) \bfPs 
\end{equation}
holds. Using the bound $ \| \bfPs - \bfPsh\|_{L^\infty(\Sigma_h)} 
\lesssim \| \bfn - \bfn_h\|_{L^\infty(\Sigma_h)} \lesssim h$ the 
estimate follows.

\subsection{Change of Domain of Integration}

The surface measure $\ds$ on $\Sigma$ is related to the surface measure $\dsh$ 
on $\Sigma_h$ by the identity  
\begin{equation}
d \sigma = |\bfB | d \sigma_h
\end{equation}
where $|\bfB|$ is the determinant of $\bfB$ which is given by
\begin{equation}
|\bfB| = \Big( \prod_{i=1}^2 (1 - \rho \kappa_i^e (1 + \rho \kappa_i^e)^{-1})\Big) \bfn \cdot \bfn_h
\end{equation}
%where ${\kappa_i^e}$, $i=1,2,$ are the nonzero eigenvalues of ${\bfkappa^e}$. 
Using this 
identity we obtain the estimate
\begin{align}
&\|1 - |\bfB| \|_{L^\infty(\Sigma_h)} 
= 
\|1 - \Big( \Pi_{i=1}^2  (1 - \rho \kappa_i^e (1 + \rho \kappa_i)^{-1})\Big) {\bfn} \cdot \bfn_h \|_{L^\infty(\Sigma_h)} 
\\
&\qquad \leq 
\|1 - \bfn \cdot \bfn_h\|_{L^\infty(\Sigma_h)}  + O( \| \rho \|_{L^\infty(\Sigma_h)} )
\lesssim h^2
\end{align}
where we finally used the bound  $2(1 - \bfn \cdot \bfn_h) = | \bfn - \bfn_h|^2 
\lesssim h^2$. In summary we have the following estimates for the determinant
\begin{equation}\label{detBbounds}
\|\, |\bfB|\,\|_{L^\infty(\Sigma_h)} \lesssim 1,
\quad
\|\, |\bfB|^{-1} \,\|_{L^\infty(\Sigma_h)} \lesssim 1,
\quad
 \|1 - |\bfB| \|_{L^\infty(\Sigma_h)} \lesssim h^2
\end{equation}

\section{Estimate of the Condition Number}

%To estimate the condition number of the stiffness matrix we 
%first prove a Poincar\'e inequality in Lemma \ref{lemmaCondB} 
%and an inverse estimate in Lemma \ref{lemmaCondC}. Then the 
%condition number estimate follow using the approach in 
%\cite{ErGu06}.

\subsection{Discrete Poincar\'e Estimates}

In this section we derive several discrete Poincar\'e estimates. 
We begin with the standard Poincar\'e inequality on $\Sigma_h$ 
for functions in $H^1(\Sigma_h)$ with average zero and a constant 
uniform in $h$ for $h$ small enough. 

Then we show estimates that essentially quantifies the improved 
control of the solution and its gradient provided by the gradient 
jump stabilization term. In order to prepare for the proof of our 
main estimates Lemma \ref{lemmaCondB} and Lemma \ref{lemmaCondBB} 
we first prove a Poincar\'e inequality for piecewise constant 
functions defined on $\mcK_h$ in Lemma \ref{lemmaCondA} and then 
in Lemma \ref{lemmaCondAA} we quantify the improved control of 
the total gradient provided by the stabilization term. 

The proof of Lemma \ref{lemmaCondA} builds on the idea of using 
a covering of $\Omega_h$ in terms of sets consisting of 
a uniformly bounded number of elements. On these sets a local 
Poincar\'e estimate holds for functions with local average zero. 
The local averages can then be approximated by a smooth function 
for which a standard Poincar\'e estimate on the exact surface can 
finally be applied. This approach is used in \cite{LeMu11} to 
prove Korn's inequality in a tubular neighborhood of a smooth 
surface. In contrast to \cite{LeMu11} our proof handles discrete 
functions and the fact that $\Omega_h$ is a polygon that changes 
with the mesh size $h$. 

\begin{lem} \label{lemmapoincaresigmah} Let 
$\lambda_h:L^2(\Sigma_h) \rightarrow \IR$ be the average 
$\lambda_h(v) = |\Sigma_h|^{-1}\int_{\Sigma_h} v d \sigma$. 
Then the following estimate holds
\begin{equation}
\| v - \lambda_h(v) \|_{\Sigma_h} 
\leq 
\| \nabla_{\Sigma_h} v \|_{\Sigma_h} 
\quad \forall v \in H^1(\Sigma_h)
\end{equation}
for $0<h\leq h_0$ with $h_0$ small enough.
\end{lem}
\begin{proof} Let $\lambda(v) = |\Sigma|^{-1}\int_{\Sigma} v d\sigma$ 
be the average of $v\in L^2(\Sigma)$. Using the fact that 
$\alpha = \lambda_h(v)$ is the constant that minimizes 
$\| v - \alpha \|_{\Sigma_h}$ and then changing coordinates 
to $\Sigma$ followed by the standard Poincar\'e estimate 
on $\Sigma$ we obtain 
\begin{align}
\| v - \lambda_h(v) \|_{\Sigma_h} &\lesssim \|v - \lambda(v^l) \|_{\Sigma_h}
\\
&\lesssim \|v^l -\lambda(v^l)\|_\Sigma 
\\
&\lesssim \| \nabla_\Sigma v^l \|_{\Sigma}
\\
&\lesssim \| \nabla_{\Sigma_h} v \|_{\Sigma_h}
\end{align}
where we mapped back to $\Sigma_h$ in the last step. This concludes 
the proof.
\end{proof}

\begin{lem} \label{lemmaCondA} Let $v$ be a piecewise constant function 
on $\mcK_h$ and let 
$\lambda_h(v)=|\Omega_h|^{-1}\int_{\Omega_h} v dx$ be the average on $\Omega_h$. Then the following estimate holds
\begin{equation}
\| v - \lambda_h(v) \|^2_{\Omega_h} \lesssim h^{-1} \sum_{F \in \mcF_h} \|[v]\|^2_F
\end{equation}
for $0<h\leq h_0$ with $h_0$ small enough.
\end{lem}
\begin{proof} 
Let $B_{\delta}(\bfx) = \{\bfy \in \IR^n: |\bfy-\bfx|_{\IR^d} < \delta\}$ 
and let $D_{h,\bfx} = B_{h}(\bfx) \cap \Sigma$ for $\bfx \in \Sigma$. 
Next we let
\begin{equation}
\mcK_{h,\bfx} = \{K \in \mcK_h, \overline{K}^l \cap D_{h,\bfx} \neq \emptyset\},
\qquad
\omega_{h,\bfx} = \cup_{K\in \mcK_{h,\bfx}} K
\end{equation} 
We note that there is a uniform bound 
$\text{card}(\mcK_{h,\bfx})\lesssim 1$ 
on the number of elements in $\mcK_{h,\bfx}$ since the mesh is 
quasiuniform. Furthermore, if $\mathcal{X}_h$ is a set of points on 
$\Sigma$ such that  $\{ D_{h,\bfx}: \bfx \in \mathcal{X}_h \}$ is 
a covering of $\Sigma$ then $\{\omega_{h,\bfx} : \bfx \in \mathcal{X}_h\}$ is a covering of $\Omega_h$. 

Next let $\chi:[0,1)\rightarrow \IR$ be smooth, nonnegative, 
have compact support, and be constant equal to $1$ in a neighborhood 
of $0$. Define $\chi_{h,x}(\bfz) = \chi(|\bfz-\bfx|_{\IR^d})/h)$ 
and 
\begin{equation}
\varphi_{h,\bfx}(\bfz) = \frac{\chi_{h,\bfx}(\bfz)}{\int_{\Sigma} 
\chi_{h,\bfx}(\bfz) d\bfz}
\end{equation}
Then $\text{supp}(\varphi_{h,\bfx})\cap \Sigma \subset D_{h,\bfx}$ 
and we have the estimates
\begin{equation}\label{eq:varphibounds}
\|\varphi_{h,\bfx}\|_{L^\infty(D_{h,\bfx})} \lesssim h^{1-d},
\qquad 
\|\nabla \varphi_{h,\bfx}\|_{L^\infty(D_{h,\bfx})} \lesssim h^{-d}
\end{equation}

On all of the sets $\omega_{h,\bfx}$ we have the local Poincar\'e 
estimate
\begin{equation}\label{eq:localpoincare}
\|a_{\bfx} - v\|^2_{\omega_{x,h}}
\lesssim h^2 \sum_{F \in \mcF_{h,\bfx}} h^{-1}\|[v]\|_F^2
\end{equation}
where $a_{\bfx} = |\omega_{h,\bfx}|^{-1}\int_{\omega_{h,\bfx}} v$ 
is the average of $v$ over $\omega_{h,\bfx}$ and $\mcF_{h,\bfx}$ is 
the set of interior faces in $\mcK_{h,\bfx}$. We finally 
let $\tilde{a} = \tilde{a}(v)$ be defined by 
\begin{equation}
\tilde{a}:\Sigma \ni \bfx \mapsto \int_{\Sigma} \varphi_{\bfx}(\bfz)v^l(\bfz) d\bfz \in \IR
\end{equation}
with average 
\begin{equation}
a = |\Sigma|^{-1}\int_\Sigma \tilde{a} d\sigma
\end{equation}
%
%$\tilde{a}:\Sigma \rightarrow \IR$ such that $\tilde{a}(\bfx) = \int_{\Sigma} \varphi_{\bfx}(\bfz)v^l(\bfz) d\bfz$ and we 
%let $a=|\Sigma|^{-1}\int_{\Sigma} \tilde{a}$ be the average of 
%$\tilde{a}$ over $\Sigma$.

With these definitions we have the estimates
\begin{align}\label{eq:CondA:aaa}
\| v - \lambda_h(v) \|^2_{\Omega_h} &\leq \| v - a \|^2_{\Omega_h} 
\\ \label{eq:CondA:bbb}
&\lesssim 
\sum_{{\bfx} \in \mathcal{X}_h} \| v - a_{\bfx} \|^2_{\omega_{h,\bfx}}
+ \| a_{\bfx} - a \|^2_{\omega_{h,\bfx}}
\\  \label{eq:CondA:ccc}
&\lesssim \sum_{\bfx \in \mathcal{X}_h} h^2 
\sum_{F \in \mcF_I(\omega_{h,\bfx})} h^{-1}\|[v]\|_F^2 
+ \sum_{\bfx \in \mathcal{X}_h} h\| a_{\bfx} - a \|^2_{D_{h,\bfx}}
\end{align}
where we used the fact that $\alpha = \lambda_h(v)$ is the 
constant that minimizes 
$\| v - \alpha \|^2_{\Omega_h}$ in (\ref{eq:CondA:aaa}) and 
the local Poincar\'e estimate (\ref{eq:localpoincare}) to 
estimate the first term and the estimate $|\omega_h| \lesssim h |D_{x,h}|$ 
together with the fact that $a_{\bfx} - a$ is a constant to estimate 
the second term in (\ref{eq:CondA:bbb}). Next we split the second 
term in (\ref{eq:CondA:ccc}) by adding and subtracting the constant 
$\tilde{a}(\bfx)$ and the function $\tilde{a}$ in each term 
in the sum as follows
\begin{align}
\sum_{\bfx \in \mathcal{X}_h} h\| a_{\bfx} - a \|^2_{D_{h,\bfx}} 
&\lesssim 
\sum_{\bfx \in \mathcal{X}_h} h\| a_{\bfx} - \tilde{a}(\bfx)\|^2_{D_{h,\bfx}}
\\ \nonumber
&\qquad \sum_{\bfx \in \mathcal{X}_h}
+h \| \tilde{a}(\bfx) - \tilde{a}\|^2_{D_{h,\bfx}}
+
\sum_{\bfx \in \mathcal{X}_h}
h \|\tilde{a} - a \|^2_{D_{h,\bfx}}
\\
&=I + II + III 
\end{align}
We proceed with estimates of Terms $I-III$.

\paragraph{Term $\bfI$} Using the fact that $a_{\bfx}-\tilde{a}(\bfx)$ 
is constant on $\omega_{h,\bfx}$, the bound 
$|\omega_{h,\bfx}| \lesssim h^n$, the definition of $\tilde{a}$, 
and the identity $a_{\bfx} = a_{\bfx} \int_{\Sigma} \varphi_x(\bfz) d\bfz 
= \int_{D_{h,\bfx}} a_{\bfx} \varphi_{\bfx}(\bfz) d\bfz$, we obtain
\begin{align}
\| a_{\bfx} - \tilde{a}(\bfx) \|^2_{\omega_{h,\bfx}} 
&\lesssim h^d |a_{\bfx} - \tilde{a}(\bfx)|^2
\\
&\lesssim h^d \left|a_{\bfx} - \int_{\Sigma_h} \varphi_{\bfx}(\bfz)v^l(\bfz)d\bfz \right|^2
\\
&\lesssim h^d \left|\int_{\Sigma_h} \varphi_{\bfx}(\bfz)(a_{\bfx} - v^l(\bfz))d\bfz \right|^2
\\ \label{lemma:a:I:a}
&\lesssim h^d \|\varphi_{\bfx}\|^2_{D_{h,\bfx}} 
\|a_{\bfx} - v^l\|^2_{D_{h,\bfx}}
\\ \label{lemma:a:I:b}
&\lesssim h \|a_{\bfx} - v^l\|^2_{D_{h,\bfx}}
\\ \label{lemma:a:I:c}
&\lesssim \|a_{\bfx} - v\|^2_{\omega_{h,\bfx}}
\\ \label{lemma:a:I:d}
&\lesssim h^2 \sum_{F \in \mcF_{h,\bfx}} h^{-1}\|[v]\|_F^2  
\end{align}
where we used Cauchy-Schwarz in (\ref{lemma:a:I:a}), the bounds (\ref{eq:varphibounds}) for $\varphi_{\bfx}$ in (\ref{lemma:a:I:b}), 
an element wise inverse inequality in (\ref{lemma:a:I:c}), and finally 
the local Poincar\'e estimate (\ref{eq:localpoincare}) in (\ref{lemma:a:I:d}). Thus we have the estimate
\begin{equation}
I = \sum_{\bfx \in \mathcal{X}_h} h\| a_{\bfx} - \tilde{a}(\bfx)\|^2_{D_{h,\bfx}}
\lesssim \sum_{\bfx \in \mathcal{X}_h} h^2 \sum_{F \in \mcF_{h,\bfx}} h^{-1}\|[v]\|_F^2
\end{equation}

\paragraph{Term $\bfI\bfI$} Using the estimate 
$|D_{h,\bfx}|\lesssim h^{d-1}$ followed by the fundamental 
theorem of calculus we obtain
\begin{align}
h \| \tilde{a}(\bfx) - \tilde{a} \|^2_{D_{h,\bfx}} 
&\lesssim h^d \| \tilde{a}(\bfx) - \tilde{a} \|^2_{L^\infty(D_{h,\bfx})} 
\\
&\lesssim h^{d+2} \| \nabla_\Sigma \tilde{a} \|^2_{L^\infty(D_{h,\bfx})}
\end{align}
For each $\bfy \in D_{h,\bfx}$ we have
\begin{align}
&\nabla_{\Sigma,\bfy} \tilde{a}(\bfy)
=
\nabla_{\Sigma,\bfy} \int_\Sigma \varphi_{\bfy}(\bfz)v^l(\bfz) d\bfz 
= 
\int_{\Sigma} (\nabla_{\Sigma,\bfy} \varphi_{\bfy})(\bfz)v^l(\bfz) d\bfz 
\\ 
&\qquad= 
\int_\Sigma (\nabla_{\Sigma,\bfy} \varphi_{\bfy})(\bfz)(v^l(\bfz) - a_{\bfy}) d\bfz
\leq
\|\nabla_{\Sigma,{\bfy}} \varphi_{\bfy} \|_{D_{h,\bfy}}\|v^l - a_{\bfy}\|_{D_{h,\bfy}}
\end{align}
and thus we find that 
\begin{align}
|\nabla_\Sigma \tilde{a}(\bfy)|_{\IR^d}^2 &\lesssim h^{-(d+1)} 
\|v^l - a_{\bfy}\|^2_{D_{h,\bfy}}
\\
&\lesssim 
h^{-(d+2)} \|v - a_{\bfy}\|^2_{\omega_{h,\bfy}}
\\
&\lesssim 
h^{-d} \sum_{F \in \mcF_{h,\bfy}} h^{-1}\|[v]\|_F^2
\end{align}
where we used the local Poincar\'e estimate (\ref{eq:localpoincare}). 
Taking the supremum over $D_{h,\bfx}$ we obtain
\begin{equation}\label{eq:maxatilde}
\| \nabla_\Sigma \tilde{a}\|^2_{L^\infty(D_{h,\bfx})}
\lesssim 
h^{-d} \sum_{F \in 2\mcF_{h,\bfx}} h^{-1}\|[v]\|_F^2
\end{equation}
where $2\mcF_{h,\bfx}=\cup_{\bfy\in D_{h,\bfx}} \mcF_{h,\bfy}$ is 
the set of interior faces contained in the set 
$2\omega_{h,\bfx} = \cup_{\bfy \in D_{h,\bfx}} \omega_{h,\bfy}$. 
Thus we conclude that 
\begin{equation}
II=\sum_{\bfx \in \mathcal{X}_h} h \| \tilde{a}(\bfx) - \tilde{a} \|^2_{D_{h,\bfx}} 
\lesssim 
\sum_{\bfx \in \mathcal{X}_h} h^2 \sum_{F \in \mcF_{h,\bfx}} 
h^{-1}\|[v]\|_F^2
\end{equation}
\paragraph{Term $\bfI\bfI\bfI$} Using the standard Poincar\'e estimate 
on $\Sigma$ we obtain
\begin{align}
III&= \sum_{\bfx \in \mathcal{X}_h} h \| \tilde{a} - a \|_{D_{h,\bfx}}^2 
\lesssim h \|\tilde{a} - a \|_{\Sigma}^2
\\ 
&\qquad
\lesssim h \|\nablas \tilde{a}\|^2_\Sigma 
\lesssim \sum_{\bfx \in \mathcal{X}_h} 
h \| \nablas \tilde{a} \|_{D_{h,\bfx}}^2
\\
&\qquad 
\lesssim 
h^{d} \sum_{\bfx \in \mathcal{X}_h} \| \nablas \tilde{a} \|_{L^\infty(D_{h,\bfx}}^2
\lesssim 
 \sum_{\bfx \in \mathcal{X}_h} \sum_{F \in 2\mcF_{h,x}} h^{-1}\|[v]\|_F^2
\end{align}
where we used (\ref{eq:maxatilde}). 

Starting from the bound (\ref{eq:CondA:ccc}) and using the 
bounds of Terms $I,II,$ and $III$ for the second term 
we arrive at 
\begin{equation}
\|  v - \lambda_h(v) \|^2_{\Omega_h} \lesssim 
\sum_{\bfx \in \mathcal{X}_h} \sum_{F \in 2\mcF_{h,x}} h^{-1}\|[v]\|_F^2
\lesssim  \sum_{F \in \mcF_{h}} h^{-1}\|[v]\|_F^2
\end{equation}
where, in the last inequality, we used the fact that there is 
a covering such that we have a uniform bound on the number of sets 
$\mcF_{h,\bfx}, \bfx \in \mathcal{X}_h$ that each edge belongs to. 
To construct such a covering we let 
$\mathcal{Y}_{h/2} = \{\bfy \in ((h/2)\mathbb{Z})^d : |\rho(\bfy)|<h/2\}$ 
and $\mathcal{X}_{h} = \{\bfx = \bfp(\bfy): \bfy \in \mathcal{Y}_{h/2}\}.$
Then $\Sigma \subset \cup_{\bfy \in \mathcal{Y}_{h/2}} B_{h/2}(\bfy)$ 
and we note that $\Sigma \cap B_{h/2}(\bfy) \subset \Sigma \cap B_h(\bfp(\bfy))= D_{h,\bfp(\bfy)}$. Thus $\{D_{h,\bfx} : \bfx \in \mathcal{X}_h\}$ is a covering of $\Sigma$. Next we note that there 
is a constant $C$ such that $(2\omega_{h,\bfx}^l)\cap \Sigma 
\subset D_{Ch,\bfx}$ since the mesh is quasiuniform. The covering 
number of $\{D_{C h,\bfx} : \bfx \in \mathcal{X}_h\}$ is uniformly 
bounded by the number of points in the set $\{\bfy \in ((h/2)\mathbb{Z})^d : 
|\bfy - \bfx|< Ch \}$, which can be estimated by $(2(2C+1))^d$.
\end{proof}

\begin{lem}\label{lemmaCondAA} The following estimate holds
\begin{equation}
h^2 \|\nabla v \|^2_{\Sigma_h} \lesssim h^2 \|\nabla_{\Sigma_h} v \|^2_{\Sigma_h} + \tn v \tn^2_{\mcF_h} \lesssim \tn v \tn_h^2 \quad \forall v \in \mcV_h
\end{equation}
for $0<h\leq h_0$ with $h_0$ small enough.
\end{lem}
\begin{proof} We have
\begin{equation}\label{eq:CondAAa}
h^2 \| \nabla v \|_{\Sigma_h}^2 
\lesssim h^2 ( \| \nabla v - \bfa \|_{\Sigma_h}^2 
+ \| \bfa \|_{\Sigma_h}^2)
\lesssim h \| \nabla v - \bfa \|_{\Omega_h}^2 
+ h^2 \| \bfa \|_{\Sigma_h}^2
\end{equation}
for all $\bfa \in \IR^3$. Choosing $\bfa$ such that  
$(\nabla v - \bfa, \bfb)_{\Omega_h}=0$ $\forall \bfb\in \IR^d$, 
it follows from Lemma \ref{lemmaCondA} that 
\begin{equation}\label{eq:CondAAb}
h \| \nabla v - \bfa \|_{\Omega_h}^2 \lesssim \tn v \tn_{\mcF_h}^2
\end{equation}
Next, to estimate $h^2 \| \bfa \|^2_{\Sigma_h}$ we first map to 
$\Sigma$ and then use the fact that $\Sigma$ is a closed surface 
followed by finite dimensionality of the constant functions to 
conclude that
\begin{align}
\| \bfa \|_{\Sigma_h} 
&\lesssim \| \bfa \|_{\Sigma} 
\lesssim \| \bfPs \bfa \|_{\Sigma} 
 \lesssim \| \bfPs \bfa \|_{\Sigma_h}
\\ 
&\qquad
\lesssim \| \bfPsh \bfa \|_{\Sigma_h} 
+ \| (\bfPs - \bfPsh) \bfa \|_{\Sigma_h}
\lesssim \| \bfPsh \bfa \|_{\Sigma_h} 
+ h \| \bfa \|_{\Sigma_h}
\end{align}
where we mapped back to the discrete surface 
$\Sigma_h$ and used the estimate 
$\|\bfPs - \bfPsh \|_{L^\infty(\Sigma_h)} 
\lesssim h$. 
Finally, for $0<h\leq h_0$ with $h_0$ sufficiently small, 
a kick back argument leads to the estimate 
\begin{equation}\label{eq:CondAAc}
h^2 \| \bfa \|^2_{\Sigma_h} 
\lesssim h^2 \| \bfPsh \bfa \|^2_{\Sigma_h} 
\end{equation}
Next, writing 
$\bfPsh \bfa = \bfPsh  \nabla v - \bfPsh  (\nabla v - \bfa )
$, we have
\begin{align}
h^2\| \bfPsh \bfa \|^2_{\Sigma_h} &
\lesssim h^2 \| \bfPsh \nabla v\|^2_{\Sigma_h} 
+  h^2 \| \bfPsh (\nabla v - \bfa) \|^2_{\Sigma_h}
\\
&\lesssim h^2\| \nabla_{\Sigma_h} v\|^2_{\Sigma_h} 
+ h  \| \nabla v - \bfa \|^2_{\Omega_h}
\\ \label{eq:CondAAd}
&\lesssim h^2 \| \nabla_{\Sigma_h} v\|^2_{\Sigma_h} + \tn v \tn^2_{\mcF_h} 
\end{align}
Combining the estimates (\ref{eq:CondAAa}), (\ref{eq:CondAAb}),  (\ref{eq:CondAAc}), and (\ref{eq:CondAAd}), we obtain the desired result.
\end{proof}

We are now ready to state and prove our Poincar\'e inequality.

\begin{lem}\label{lemmaCondB} 
The following estimate holds
\begin{align}\label{poincare}
 \| v \|_{\Omega_h} &\lesssim h^{1/2} \tn v \tn_h \quad \forall v \in V_h
\end{align}
for $0<h\leq h_0$ with $h_0$ small enough.
\end{lem}
\begin{proof} Using the same notation as in Lemma \ref{lemmaCondA}
we first show that there is a constant $C$ such that 
for each $\bfx\in \Sigma$ and $h$, $0<h\leq h_0$, 
there exists an element $K_{h,\bfx} \in \mcK_{h,\bfx}$ 
such that 
\begin{equation}
C h^2 \leq |K_{h,\bfx} ^l\cap \Sigma|
\end{equation}
We say that such an element has a large intersection with $\Sigma_h$.
Assume that there is no such element. Then there is a sequence $h_n \rightarrow 0$ such that 
\begin{equation}
|K^l\cap \Sigma|\leq \frac{h_n^2}{n}
\quad \forall K\in \mcK_{h,\bfx}, \quad n = 1,2,3,\dots 
\end{equation}
Since there is a uniform bound on the number of elements in 
$\mcK_{h,\bfx}$ we obtain the estimate
\begin{equation}
h_n^{-2}|\omega_h^l \cap \Sigma| \lesssim n^{-1}
\end{equation}
which is a contradiction since $h_n^2 \sim |D_{h,\bfx}| 
\lesssim |\omega_h^l\cap \Sigma|$ since $D_{h,\bfx}
\subset \omega_h^l\cap\Sigma$. Furthermore, the 
following estimate holds
\begin{equation}\label{eq:CondBa}
\|v\|^2_{\omega_{h,\bfx}}\lesssim 
h \|v\|^2_{K_{h,\bfx}\cap \Sigma_h} 
+ h^3 \|\bfn_h \cdot \nabla v\|^2_{K_{h,\bfx}\cap \Sigma_h} 
+ h^3 \tn v \tn^2_{\mcF_{h,\bfx}}
%+ \sum_{F \in \mcF_{h,\bfx}} h^3 \|[\bfn_F \cdot \nabla v]\|_F^2
\end{equation}
where we introduced the notation
\begin{equation}
\tn v \tn^2_{\mcF_{h,\bfx}} = \sum_{F \in \mcF_{h,\bfx}} 
\|[\bfn_F \cdot \nabla v]\|_F^2
\end{equation}
To prove (\ref{eq:CondBa}), let $K_1$ and $K_2$ be two elements 
sharing a common face $F$. Then we have the following identity
\begin{equation}\label{eq:CondB:aa}
v_2 = v_1 + [\bfn_F \cdot \nabla v] \bfn_F \cdot (\bfx -\bfx_F)
\end{equation}
where $\bfx_F$ is the center of gravity of the face $F$. Thus
\begin{align}
\| v_2 \|^2_{K_2} &\lesssim \| v_1 \|^2_{K_2} + \| [\bfn_F \cdot \nabla v] \bfn_F \cdot (\bfx -\bfx_F)\|^2_{K_2}
\\ \label{CondTerm2}
&\lesssim \| v_1 \|^2_{K_1} + h^{3} \| [\bfn_F \cdot \nabla v] \|^2_F
\end{align}
Iterating this bound and summing over all elements in $\mcK_{h,\bfx}$ 
we obtain
\begin{equation}\label{eq:CondB:b}
\|v\|^2_{\omega_{h,\bfx}} 
%\lesssim \|v\|^2_{K_{h,\bfx}} 
%+ \sum_{F \in \mcF_{h,\bfx}} h^3 \|[\bfn_F \cdot \nabla v]\|_F^2
\lesssim \|v\|^2_{K_{h,\bfx}} + h^3 \tn v \tn^2_{\mcF_{h,\bfx}}
\end{equation}
Here $\| v \|^2_{K_{h,\bfx}}$ may be estimated using the inverse bound
\begin{equation}\label{eq:CondB:c}
\| v \|^2_{K_{h,\bfx}} \lesssim h \| v \|^2_{K_{h,\bfx}\cap \Sigma_h} 
+ h^3 \|\bfn_h \cdot \nabla v \|^2_{K_{h,\bfx} \cap \Sigma_h}
\end{equation}
which holds due to quasiuniformity and the fact that the intersection 
with $\Sigma_h$ of $K_{h,\bfx}$ satisfies 
$h^2 \lesssim |K_{h,\bfx}^l \cap \Sigma|\sim|K_{h,\bfx} \cap \Sigma_h|$. Combining (\ref{eq:CondB:b}) and (\ref{eq:CondB:c}) we obtain (\ref{eq:CondBa}).

Using a covering $\{\omega_{h,\bfx} : \bfx \in \mathcal{X}_h \}$ 
of $\Omega_h$ with a uniformly bounded covering number as constructed in Lemma \ref{lemmaCondA}, we obtain 
\begin{align}
\|v\|_{\Omega_h}^2 
&\lesssim 
\sum_{\bfx \in \mathcal{X}_h} \|v\|_{\omega_{h,\bfx}}^2
\\
&\lesssim \sum_{\bfx \in \mathcal{X}_h}  h \|v\|^2_{K_{h,\bfx}\cap \Sigma_h} 
+ h^3 \|\bfn_h \cdot \nabla v\|^2_{K_{h,\bfx}\cap \Sigma_h} 
+ h^3 \tn v \tn^2_{\mcF_{h,\bfx}}
%\\ \nonumber
%&\qquad + \sum_{F \in \mcF_{h,\bfx}} h^3 \|[\bfn_F \cdot \nabla v]\|_F^2
\\
&\lesssim h \| v \|^2_{\Sigma_h} + h^3 \|\nabla v \|^2_{\Sigma_h} 
 + h^3 \tn v \tn^2_{\mcF_h}
 %+ \sum_{F \in \mcF_{h}} h^3 \|[\bfn_F \cdot \nabla v]\|_F^2
\\ \label{eq:CondA:a}
&\lesssim h \| v \|^2_{\Sigma_h} 
+ \Big( h^3 \|\nabla_{\Sigma_h} v \|^2_{\Sigma_h} + h \tn v \tn^2_{\mcF_h}
\Big)
+ h^3 \tn v \tn^2_{\mcF_h}
\\ \label{eq:CondA:b}
&\lesssim h \| \nabla_{\Sigma_h} v \|^2_{\Sigma_h} 
+ \Big( h^3 \|\nabla_{\Sigma_h} v \|^2_{\Sigma_h} + h \tn v \tn^2_{\mcF_h}
\Big)
+ h^3 \tn v \tn^2_{\mcF_h}
\\
&\lesssim h \tn v \tn^2_h
\end{align}
where we used the Poincar\'e inequality, see Lemma 
\ref{lemmapoincaresigmah}. This concludes the proof.
\end{proof}

We conclude this section with a version of the previous 
lemma which involves the $L^2$ norm instead of the energy 
norm.
\begin{lem}\label{lemmaCondBB} 
The following estimate holds
\begin{align}\label{poincareL2}
 \| v \|^2_{\Omega_h} &\lesssim h \| v \|^2_{\Sigma_h} 
 + h \tn v \tn^2_{\mcF_h} \quad \forall v \in V_h
\end{align}
for $0<h\leq h_0$ with $h_0$ small enough.
\end{lem}
\begin{proof} Starting from (\ref{eq:CondA:b}) we have 
\begin{equation}\label{eq:CondBB:a}
\|v\|_{\Omega_h}^2 \lesssim  h \| v \|^2_{\Sigma_h} 
+ h^3 \|\nabla_{\Sigma_h} v \|^2_{\Sigma_h} + h \tn v \tn^2_{\mcF_h}
\end{equation}
and thus we need to estimate $h^3 \|\nabla_{\Sigma_h} v \|^2_{\Sigma_h}$. 
Using the same notation as above we obtain
\begin{align}
h^3 \|\nabla_{\Sigma_h} v \|^2_{\Sigma_h} 
&\lesssim 
\sum_{\bfx \in \mcX_h} 
h^3 \|\nabla_{\Sigma_h} v \|^2_{\Sigma_h \cap \omega_{h,\bfx}} 
\\
&\lesssim 
\sum_{\bfx \in \mcX_h} 
h^3 \|\nabla_{\Sigma_h} v \|^2_{\Sigma_h \cap K_{h,\bfx}}
+ 
h^3 \|\nabla_{\Sigma_h} v \|^2_{\Sigma_h \cap (\omega_{h,\bfx} 
\setminus K_{h,\bfx})} 
\\
&\lesssim \label{eq:CondBB:c}
\sum_{\bfx \in \mcX_h} 
h \| v \|^2_{\Sigma_h \cap K_{h,\bfx}}
+ 
\underbrace{h^2 \|\nabla_{\Sigma_h} v - \nabla_{\Sigma_{h,\bfx}} v \|^2_{\omega_{h,\bfx}}}_{I_{\bfx}} 
\\ \nonumber
&\qquad \qquad + \underbrace{h^2 \|\nabla_{\Sigma_{h,\bfx}} v \|^2_{\omega_{h,\bfx}}}_{II_{\bfx}}
\end{align}
where we used the fact that $K_{h,\bfx}$ has a large intersection 
with $\Sigma_h$ so that an inverse estimate holds for the tangential derivative and we also added and subtracted $\nabla_{\Sigma_{h,\bfx}} v$
where $\nabla_{\Sigma_{h,\bfx}}$ is the tangential gradient to 
$\Sigma_h \cap K_{h,\bfx}$. We proceed with estimates of $I_{\bfx}$ and 
$II_{\bfx}$.

\paragraph{Term $I_{\bfx}$} Using the definition of the tangential derivative we obtain the estimate
\begin{equation}
\|\nabla_{\Sigma_h} v - \nabla_{\Sigma_{h,\bfx}} v \|^2_{K} 
\lesssim  | \bfn_{h,K} - \bfn_{h,K_{h,\bfx}}|_{\IR^d} \| \nabla v \|^2_K\quad \forall K \in \mcK_{h,\bfx} 
\end{equation}
where $\bfn_{h,K}$ is the (constant) normal associated 
with element $K$. Now
\begin{equation}
|\bfn_{h,K} - \bfn_{h,K_{h,\bfx}}|_{\IR^d} \lesssim h, 
\quad \forall K \in \mcK_{h,\bfx}
\end{equation}
since
\begin{align}
|\bfn_{h,K} - \bfn_{h,K_{h,\bfx}}|_{\IR^d} &\leq 
|\bfn_{h,K} - \bfn(\bfy)|_{\IR^d} 
+ |\bfn_{h,K_{h,\bfx}}- \bfn(\bfx)|_{\IR^d}
+ |\bfn(\bfy) - \bfn(\bfx)|_{\IR^d}
\end{align} 
for any $\bfy \in K^l\cap \Sigma$. Using (\ref{eq:rhobounds}) we 
conclude that the first two terms are $O(h)$ and using the 
fundamental theorem of calculus we have the estimate
\begin{equation}
|\bfn(\bfy) - \bfn(\bfx)|_{\IR^d}
=|\nabla \rho(\bfy) - \nabla \rho(\bfy)|_{\IR^d}
\lesssim |\bfx - \bfy|_{\IR^d} \|\nabla \otimes \nabla \rho \|_{L^{\infty}(\Sigma)}
\lesssim h 
\end{equation}
for the third term. Thus we have the estimate
\begin{equation}
\|\nabla_{\Sigma_h} v - \nabla_{\Sigma_{h,\bfx}} v \|^2_{K} 
\lesssim h^2 \| \nabla v \|^2_K\quad \forall K \in \mcK_{h,\bfx} 
\end{equation}
which gives
\begin{equation}\label{eq:CondBB:I}
I_{\bfx} = h^2 \|\nabla_{\Sigma_h} v - \nabla_{\Sigma_h\cap K_{h,\bfx}} v \|^2_{\omega_{h,\bfx}} 
\lesssim h^4 \| \nabla v \|^2_{\omega_{h,\bfx}}
\lesssim h^2 \| v \|^2_{\omega_{h,\bfx}} 
\end{equation}
where we used an inverse estimate at last.

\paragraph{Term $II_{\bfx}$}
We let $\nabla_{\Sigma_{h,\bfx}}$ act on identity 
(\ref{eq:CondB:aa}), which gives
\begin{equation}
\nabla_{{\Sigma_{h,\bfx}}} v_2 = \nabla_{\Sigma_{h,\bfx}} v_1 
+ [\bfn_F \cdot \nabla v] \bfP_{\Sigma_{h,\bfx}} \bfn_F 
\end{equation}
and thus we have the estimate
\begin{equation}
\|\nabla_{{\Sigma_{h,\bfx}}} v_2 \|^2_{K_2} \lesssim 
\| \nabla_{\Sigma_{h,\bfx}} v_1 \|^2_{K_1} 
+ h \|[\bfn_F \cdot \nabla  v]\|^2_F 
\end{equation}
Iterating this bound and summing over all the elements in $\mcK_{h,\bfx}$, 
again using the fact that the number of elements in $\mcK_{h,\bfx}$ is uniformly bounded, we arrive at
\begin{equation}\label{eq:CondBB:aa}
\| \nabla_{{\Sigma_{h,\bfx}}} v \|^2_{\omega_{h,\bfx}} 
\lesssim 
\| \nabla_{\Sigma_{h,\bfx}} v \|^2_{K_{h,\bfx}} 
+ h \tn v \tn^2_{\mcF_{h,\bfx}} 
\end{equation}
Using (\ref{eq:CondBB:aa}) we obtain
\begin{align}
II_{\bfx} &=h^2 \|\nabla_{\Sigma_{h,\bfx}} v \|^2_{\omega_{h,\bfx}}
\\
&\lesssim 
h^2 \|\nabla_{\Sigma_{h,\bfx}} v \|^2_{K_{h,\bfx}}
+ h^3 \tn v \tn^2_{\mcF_{h,\bfx}}
\\
&\lesssim \label{eq:CondBB:II}
h \| v \|^2_{\Sigma_h \cap K_{h,\bfx}}
+ h^3 \tn v \tn^2_{\mcF_{h,\bfx}}
\end{align}
where we used the estimate 
\begin{equation}
h^2 \|\nabla_{\Sigma_{h,\bfx}} v \|^2_{K_{h,\bfx}}
\lesssim h^3 \|\nabla_{\Sigma_{h,\bfx}} v \|^2_{\Sigma_h \cap K_{h,\bfx}}
\lesssim h \| v \|^2_{\Sigma_h \cap K_{h,\bfx}}
\end{equation}
which holds since element $K_{h,\bfx}$ has a large intersection 
with $\Sigma_h$.

Collecting the estimates (\ref{eq:CondBB:c}), (\ref{eq:CondBB:I}), 
and (\ref{eq:CondBB:II}) we have
\begin{align}
h^3 \|\nabla_{\Sigma_h} v \|^2_{\Sigma_h} 
&\lesssim 
\sum_{\bfx \in \mcX_h} 
h \| v \|^2_{\Sigma_h \cap K_{h,\bfx}}
+ 
h \| v \|^2_{\omega_{h,\bfx}}
+ h^3 \tn v \tn^2_{\mcF_{h,\bfx}}
\\ \label{eq:CondBB:d}
&\lesssim h \| v \|^2_{\Sigma_h} + h^2 \| v \|^2_{\Omega_{h}} 
+ h^3 \tn v \tn^2_{\mcF_h}
\end{align}
Combining (\ref{eq:CondBB:a}) and (\ref{eq:CondBB:d}) yields
\begin{equation}
\|v\|_{\Omega_h}^2 \lesssim  h \| v \|^2_{\Sigma_h} 
+  h^2 \| v \|^2_{\Omega_{h}} +  h \tn v \tn^2_{\mcF_h}
\end{equation}
and the lemma follows for $0<h \leq h_0$, with $h_0$ small enough, 
using a kick back argument.
\end{proof}

\subsection{Inverse Estimate}

Here we derive the inverse inequality needed in the proof of the condition 
number estimate.

\begin{lem}\label{lemmaCondC} The following estimate holds
\begin{align}\label{inverse}
 \tn v \tn_h &\lesssim h^{-3/2} \| v \|_{\Omega_h} \quad \forall v \in {V}_h
\end{align}
for $0<h\leq h_0$ with $h_0$ small enough.
\end{lem}
\begin{proof} Using the fact that $\nabla v|_K$ is constant we note that 
$\|\nabla v \|^2_{\Sigma_h \cap K} \lesssim h^{-1} \|\nabla v \|^2_{K}$, which leads to the estimate
\begin{equation}
\| \bfPsh \nabla v\|^2_{\Sigma_h} \lesssim \| \nabla v \|^2_{\Sigma_h}
\lesssim  h^{-1} \| \nabla v \|^2_{\Omega_h} \lesssim h^{-3} \| v \|_{\Omega_h}^2
\end{equation}
where we used an element wise inverse inequality at last. Furthermore, 
using standard inverse inequalities, we obtain the following estimate 
for the jump term
\begin{equation}
\tn v \tn^2_{\mcF_h} \lesssim \sum_{K \in \mcK_h} \| \bfn_K \cdot \nabla v \|^2_{\partial K}
\lesssim \sum_{K \in \mcK_h} h^{-1} \| \nabla v \|^2_K \lesssim h^{-3} \| v \|^2_{\Omega_h}
\end{equation}
\end{proof}

\subsection{Condition Number Estimate}
To derive an estimate of the condition number of the stiffness matrix 
we use the Poincar\'e inequality in Lemma \ref{lemmaCondB} and the 
inverse estimate in Lemma \ref{lemmaCondC} together with the approach 
in \cite{ErGu06}.

Let $\{\varphi_i\}_{i=1}^N$ be the standard piecewise linear basis 
functions associated with the nodes in $\mcK_h$ and let $\mcA$ be 
the stiffness matrix with elements $a_{ij} = A_h(\varphi_i,\varphi_j)$. 
We recall that the condition number is defined by
\begin{equation}
\kappa(\mcA) = | \mcA |_{\IR^N} |\mcA^{-1} |_{\IR^N}
\end{equation}
where $|X|^2_{\IR^N} = \sum_{i=1}^N X_i^2$ for $X \in \IR^N$ and $|\mcA|_{\IR^N} = \sup_{|X|_{\IR^N} =1} |\mcA X|_{\IR^N}$
for $\mcA \in \IR^{N\times N}$. The expansion $v = \sum_{i=1}^N V_i \varphi_i$ defines an
isomorphism that maps $ v \in \mcV_h$ to $V \in \IR^N$ and satisfies 
the following well known estimates 
\begin{equation}\label{rneqv}
c h^{-d/2} \| v \|_{\Omega_h} \leq | V |_{\IR^N} \leq C h^{-d/2}\| v \|_{\Omega_h}
\end{equation}

\begin{thm} The following estimate of the condition number of the 
stiffness matrix holds
\begin{equation}
\kappa( \mcA )\lesssim h^{-2}
\end{equation}
for $0<h\leq h_0$ with $h_0$ small enough.
\end{thm}
\begin{proof} We need to estimate $| \mcA |_{\IR^N}$ and 
$|\mcA^{-1}|_{\IR^N}$. Starting with 
$| \mcA |_{\IR^N}$ we have
\begin{align}
|\mcA V|_{\IR^N} &= \sup_{W \in \IR^N } \frac{(W,\mcA V)_{\IR^N}}{| W |_{\IR^N}}
\\
%&= \sup_{w \in \mcV_h }  \frac{A_h(v,w)}{\tn w \tn_h} \frac{\tn w \tn_h}{\| w \|_{\Omega_h}} %\frac{\| w \|_{\Omega_h}}{| W |_N}
%\\
&= \sup_{w \in \mcV_h }  \frac{A_h(v,w)}{\tn w \tn_h} \frac{\tn w \tn_h}{| W |_{\IR^N}}
\\
&\lesssim h^{d-3}| V|_{\IR^N}
\end{align}
where we used the estimate
\begin{equation}
\tn w \tn_h \lesssim h^{-3/2} \| w \|_{\Omega_h} 
\lesssim h^{(d-3)/2}|W|_{\IR^N}
\end{equation}
together with (\ref{inverse}) and (\ref{rneqv}). Thus 
\begin{equation}\label{Aest}
|\mcA|_{\IR^N} \lesssim h^{d-3}
\end{equation}

Next we turn to the estimate of $|\mcA^{-1}|_{\IR^N}$.
%and note that
%\begin{equation}
%|A^{-1}|_N = \sup_{ \in \IR^N } \frac{ |A^{-1} V|}{|V|} = \sup_{W \in \IR^N } \frac{ |W|}{|A W|}
%\end{equation}
Using (\ref{rneqv}) and (\ref{poincare}), we get
\begin{multline}
|V|^2_{\IR^N} \lesssim h^{-d} \| v \|^2_{\Omega_h} \lesssim h^{1-d} \tn v \tn_h^2
\\
\lesssim h^{1-d} A_h(v,v) = h^{1-d} (V, \mcA V)_{\IR^N} \lesssim h^{1-d} |V|_{\IR^N} |\mcA V|_{\IR^N}
\end{multline}
and thus we conclude that $|V|_{\IR^N} \lesssim h^{1-d}|\mcA V|_{\IR^N}$. Setting $ V = \mcA^{-1} W$ we obtain
\begin{equation}\label{Ainvest}
|\mcA^{-1}|_{\IR^N} \lesssim h^{1-d}
\end{equation}
Combining estimates (\ref{Aest}) and (\ref{Ainvest}) of $|\mcA|_{\IR^N}$ 
and $|\mcA^{-1}|_{\IR^N}$ the theorem follows.
\end{proof}

\section{A Priori Error Estimates}
In this section we derive a priori error estimates in the energy and 
$L^2$-norms. The main technical difficulty is to handle the fact that 
the surface is approximated by a discrete surface. Our approach 
essentially follows \cite{De09}, \cite{DeDz07}, and \cite{OlReGr09}.

\subsection{Interpolation Error Estimates}

In order to define an interpolation operator we note that the extension 
$v^e$ of $v \in H^s(\Sigma)$ satisfies the stability estimate
\begin{equation}\label{eq:ext_stab}
\| v^e \|_{s,U_{\delta}(\Sigma)} \lesssim \delta^{\frac12}  \| v \|_{s,\Sigma},\quad s=0,1,2, \quad 0<\delta\leq \delta_0
\end{equation}
with constant only dependent on the curvature of the surface $\Sigma$. 
%The above dependence on $\delta$ 
%can be obtained by mapping $U_\delta(\Sigma)$ to some reference 
%shell where both the diameter and the thickness are fixed. 
%On this domain the standard result for extension operators $\| E v \|_{s,%\Omega_h} \lesssim \| v \|_{s,\Sigma}$ holds 
%and \eqref{eq:ext_stab} follows by scaling back to the physical domain %noting that the thickness, in the direction normal to $\Sigma$ is $\delta$. 

We let $\pi_h: L^2(\Omega_h) \rightarrow \mcV_h|_{\Sigma_h}$ denote the standard 
Scott-Zhang interpolation operator and recall the interpolation error estimate
\begin{equation}\label{interpolstandard}
\| v - \pi_h v \|_{m,K} \leq C h^{2-m} \| v \|_{\mcN(K)}, \quad m = 1,2
\end{equation}
where $\mcN(K)\subset \Omega_h$ is the union of the neighboring elements of $K$. We 
also define an interpolation operator 
$\pi_h^l:L^2(\Sigma) \rightarrow (\mcV_h|_{\Sigma_h})^l$ 
as follows
\begin{equation}\label{pihl}
\pi_h^l v =  ( (\pi_{h} v^e) |_{\Sigma_h})^l
\end{equation}

Introducing the energy norm $\tn \cdot \tn_\Sigma$ associated with the 
exact surface 
%and the energy norm $\tn \cdot \tn_{\mcF_h}$  associated with the 
%jump terms
\begin{equation}
\tn v \tn^2_{\Sigma} = a(v,v) %\quad \tn v \tn_{\mcF_h}^2 = j_h(v,v)
\end{equation}
we have the following approximation property.
\begin{lem}\label{lem:approx} The following estimate holds
\begin{equation}\label{interpol}
\tn u - \pi_h^l u \tn_\Sigma^2 + \tn u^e - \pi_h u^e \tn^2_{\mcF_h} 
\lesssim h^2 \| u \|^2_{2,\Sigma}
\end{equation}
for $0<h\leq h_0$ with $h_0$ small enough.
\end{lem}
\begin{proof} We first recall the element wise trace inequality
\begin{equation}\label{eq:trace}
\| v \|^2_{\Sigma_h \cap K} \lesssim h^{-1} \| v \|_K^2 
+ h \| \nabla v \|_K^2
\end{equation}
which holds with a uniform constant independent of the intersection, 
see Lemma 4.2 in  \cite{HaHaLa04}. To estimate the first term we change domain 
of integration from $\Sigma$ to $\Sigma_h$ and then use the trace inequality (\ref{eq:trace}) as follows
\begin{align}
\tn u - \pi_h^l u \tn_\Sigma^2 &= \int_\Sigma |\nablas(u - \pi_h^l u)|^2 \ds
\\
&=\int_{\Sigma_h} |\bfB^{-T}\nablash(u^e- \pi_h u)|^2 |\bfB| \dsh
\\
&\lesssim \sum_{K \in \mcK_h} \| \nabla (u^e- \pi_h u) \|^2_{K\cap \Sigma_h}
\\
&\lesssim\sum_{K \in \mcK_h} h^{-1} \| u^e- \pi_h u \|_{1,K}^2 + h \| u^e- \pi_h u \|_{2,K}^2
\\
&\lesssim \sum_{K \in \mcK_h} h \|u^e\|^2_{2,\mcN(K)}
\\
&\lesssim h \|u^e\|^2_{2,U_{\delta}(\Sigma)}
\\
&\lesssim h \delta \| u \|^2_{2,\Sigma}
\end{align}
where we used the interpolation estimate (\ref{interpolstandard}) 
followed by the stability estimate (\ref{eq:ext_stab}) for the 
extension operator. Observing that $\Omega_h \subset U_{\delta}(\Sigma)$ with $\delta\sim h$ we arrive at
\begin{equation}
\tn u - \pi_h^l u \tn_\Sigma^2 \lesssim h^2 \| u \|^2_{2,\Sigma}
\end{equation}
The second term can be directly estimated using the elementwise trace inequality followed by (\ref{interpolstandard}).
\end{proof}

\subsection{Error Estimates}

\begin{thm}\label{thmenergy} The following a priori error estimate holds
\begin{equation}\label{eqenergy}
\tn u - u_h^l \tn^2_{\Sigma} + \tn u^e - u_h \tn_{\mcF}^2 \lesssim 
h^2 \|f\|^2_\Sigma
\end{equation}
for $0<h\leq h_0$ with $h_0$ small enough.
\end{thm}
\begin{proof} Adding and subtracting an interpolant $\pi_h^l u$, defined 
by (\ref{pihl}), and $\pi_h u^e$, and using the triangle inequality we have
\begin{align}
 \tn u - u_h^l \tn_{\Sigma}^2 + \tn u^e - u_h \tn_{\mcF_h}^2 
&\lesssim \tn u^e - \pi_h^l u \tn_{\Sigma}^2 + \tn u^e - \pi_h u \tn_{\mcF_h}^2
\\ \nonumber
&\qquad+
\tn \pi_h^l u - u_h^l \tn^2_{\Sigma} + \tn \pi_h u^e - u_h \tn^2_{\mcF_h}
\end{align}
Here the first two terms can be immediately estimated using the interpolation 
error estimate (\ref{interpol}). For the third and fourth we have the following 
identity
\begin{align}\nonumber
&a(\pi_h^l u-u_h^l,\pi_h^l u - u_h^l) + j(\pi_h u^e -u_h,\pi_h u^e-u_h)
\\\nonumber
&\qquad =a(\pi_h^l u - u,\pi_h^l u-u_h^l) + j(\pi_h u^e - u^e,\pi_h u^e -u_h) 
+ l( \pi_h^l u-u_h^l) 
\\ \nonumber
&\qquad \qquad -  a(u_h^l,\pi_h^l u-u_h^l) - j(u_h,\pi_h u^e-u_h)
\\\nonumber
&\qquad= a(\pi_h^l u - u,\pi_h^l u-u_h^l) + j(\pi_h u^e - u^e,\pi_h u^e -u_h) 
\\ \nonumber
&\qquad \qquad + l( \pi_h^l u-u_h^l) - l_h(\pi_h u^e - u_h) 
\\ \nonumber
&\qquad \qquad 
+ a_h(u_h,\pi_h u^e - u_h) - a(u_h^l,\pi_h^l u - u_h^l)
\end{align}
Estimating the right hand side we obtain
\begin{align}\nonumber
\left(\tn \pi_h^l u-u_h^l \tn^2_{\Sigma} + \tn u^e - u_h \tn^2_{\mcF_h} \right)^{1/2} 
&\lesssim \left( \tn u - \pi_h^l u \tn^2_\Sigma + \tn u^e - \pi_h u^e \tn^2_{\mcF_h} \right)^{1/2}
\\ \nonumber 
&\qquad + \sup_{v \in \mcV_h} \frac{l( v^l) - l_h(v)}{\tn v^l \tn_\Sigma}
\\ \nonumber 
&\qquad + \sup_{v \in \mcV_h} \frac{a_h^l(u_h^l,v^l) - a_h(u_h,v)}{\tn v^l \tn_\Sigma}\nonumber
\\
&=I + II + III \nonumber
\end{align} 
\paragraph{Term $\bfI$} The first term can be directly estimated using the 
interpolation inequality (\ref{interpol}). 
\paragraph{Term $\bfI\bfI$} Changing domain of integration we obtain the estimate
\begin{align}
l(v^l) - l(v)&= \int_\Sigma f v^l \ds - \int_{\Sigma_h} f^e v \dsh
\\
&=\int_{\Sigma_h} f^e v |\bfB| \dsh - \int_{\Sigma_h} f^e v \dsh
\\
&=%\int_{\Sigma_h} (f^e - f_h) v |\bfB| \dsh 
\int_{\Sigma_h} f^e v (|\bfB| - 1 ) |\bfB|^{-1} |\bfB| \dsh
\\
&\lesssim  h^2 \| f^e \|_{\Sigma_h}\| v \|_{\Sigma_h}
\\
&\lesssim  h^2 \| f^e \|_{\Sigma_h}\| \nabla_{\Sigma_h} v \|_{\Sigma_h}
\\
&\lesssim h^2 \| f \|_{\Sigma}\| \nabla_{\Sigma} v \|_{\Sigma}
\end{align}
where we used the estimates (\ref{detBbounds}), 
the Poincar\'e inequality in Lemma \ref{lemmapoincaresigmah} 
on $\Sigma_h$, and finally we mapped from $\Sigma_h$ to $\Sigma$. 

\paragraph{Term $\bfI\bfI\bfI$} Changing domain of integration we obtain
\begin{align}\nonumber
&\int_{\Sigma} \nablas u_h^l \cdot \nablas v^l \ds 
      - \int_{\Sigma_h} \nablash u_h \cdot \nablash v \dsh
\\
&\qquad = \int_{\Sigma_h} \bfB^{-T} \nablash u_h \cdot \bfB^{-T}  \nablash v |\bfB| \dsh
- \int_{\Sigma_h} \nablash u_h \cdot \nablash v \dsh
\\
&\qquad = \int_{\Sigma_h} (\bfB^{-1} \bfB^{-T} - |\bfB|^{-1}\bfI) \nablash u_h \cdot \nablash v |\bfB| \dsh
\\
&\qquad = \int_{\Sigma_h} (\bfI - |\bfB|^{-1}\bfB \bfB^T) \bfB^{-T}\nablash u_h 
\cdot \bfB^{-T} \nablash v |\bfB| \dsh
\\
&\qquad \leq \| \bfI - |\bfB|^{-1} \bfB \bfB^T \|_{L^\infty(\Sigma_h)} 
\tn u_h^l \tn_\Sigma \tn v^l \tn_\Sigma
\\
&\qquad \lesssim h^2 \tn u_h^l \tn_\Sigma \tn v^l \tn_\Sigma
\\
&\qquad \lesssim h^2 \| f \|_\Sigma \tn v^l \tn_\Sigma
\end{align}
where at last we used the stability estimate
\begin{equation}
\tn u_h^l \tn_\Sigma 
\lesssim  
\tn u_h \tn_{\Sigma_h} 
\lesssim 
\|f^e \|_{\Sigma_h}
\lesssim 
\|f\|_{\Sigma}
\end{equation}
for the method. Furthermore, we used (\ref{Bestimates}) and (\ref{detBbounds}) to show the following 
estimate
\begin{align}
&\| \bfI - |\bfB|^{-1}\bfB \bfB^T \|_{L^\infty(\Sigma_h)} 
= \||\bfB|^{-1} (|\bfB| \bfI - \bfB \bfB^T) \|_{L^\infty(\Sigma_h)}
\\
&\qquad \leq 
\| |\bfB|^{-1}\|_{L^\infty(\Sigma_h)} \left(\| |\bfB| - 1\|_{L^\infty(\Sigma_h)} 
+  \|\bfI - \bfB \bfB^T \|_{L^\infty(\Sigma_h)} \right) \lesssim h^2
\end{align}
Finally, collecting the estimates of Terms $I-III$ the proof follows.
\end{proof}

\begin{thm} The following a priori error estimate holds
\begin{equation}
\| u - u_h^l \|_\Sigma \lesssim h^2 \| f \|_{\Sigma}
\end{equation}
for $0<h\leq h_0$ with $h_0$ small enough.
\end{thm}
\begin{proof}
Recall that $u_h$ satisfies $\int_{\Sigma_h} u_h d\sigma_h = 0$ 
and define $\tilde{u}_h \in V_h$ such that
\begin{equation}
\tilde{u}_h = u_h - |\Sigma|^{-1}\int_\Sigma u_h^l d\sigma
\end{equation}
Then $\int_\Sigma \tilde{u}_h^l d\sigma = 0$ and we have the estimate
\begin{equation}
\|u - u_h^l \|_\Sigma \leq \|u - \tilde{u}_h^l \|_\Sigma
+
\|\tilde{u}_h - {u}_h^l \|_\Sigma = I + II
\end{equation}
\paragraph{Term $\bfI$}
Let $\phi$ be the solution to the dual problem $-\Delta_\Sigma \phi = \psi \in L^2(\Sigma)/\IR$. Then it follows from the Lax-Milgram lemma that 
there exists a unique solution in $H^1(\Sigma)/\IR$ and we also have 
the elliptic regularity estimate $\| \phi \|_{2,\Sigma} \lesssim \|\psi\|_\Sigma$. In order to estimate $\|u - \tilde{u}_h^l\|_\Sigma$ we multiply 
the dual problem by $u - \tilde{u}_h^l$, integrate using Green's 
formula, and add and subtract suitable terms 
\begin{align}
(u - \tilde{u}_h^l, \psi ) &= a(u - \tilde{u}_h^l,\phi)
\\
&=a(u - \tilde{u}_h^l,\phi - \pi_h^l \phi ) 
+ a(u - \tilde{u}_h^l,\pi_h^l \phi )
\\
&=a(u - \tilde{u}_h^l,\phi - \pi_h^l \phi ) 
+ \left( l(\pi_h^l \phi ) - l_h(\pi_h \phi^e ) \right) 
\\ \nonumber
&\qquad + \left( a_h(\tilde{u}_h, \pi_h \phi^e) - a(\tilde{u}_h^l,\pi_h^l \phi )\right)
+ j_h(\tilde{u}_h-u^e,\pi_h \phi^e - \phi^e)
\end{align}
These terms may now be estimated using Cauchy-Schwarz, the energy norm estimate (\ref{eqenergy}), together with the estimates of terms $II$ 
and $III$ in the proof of Theorem \ref{thmenergy}. We note in particular that
\begin{equation}
|a_h(\tilde{u}_h, \pi_h \phi^e) - a(\tilde{u}_h^l,\pi_h^l \phi )|\lesssim h^2 \tn \tilde{u}_h^l \tn_\Sigma \tn \pi_h^l\phi \tn_\Sigma
\lesssim h^2 
\end{equation}
Here the first term is estimated by observing that 
\begin{equation}
\tn \tilde{u}_h^l \tn_\Sigma = \tn {u}_h^l \tn_\Sigma 
\lesssim \tn u_h \tn_{\Sigma_h} \lesssim \|f\|_{\Sigma_h}
\end{equation}
and the second term $\tn \pi_h^l\phi \tn_\Sigma$ using 
Lemma \ref{lem:approx}.

\paragraph{Term $\bfI\bfI$} Using the fact that 
$\int_{\Sigma_h} u_h d\sigma_h = 0$ we obtain
\begin{align}
&\| u_h^l - \tilde{u}_h^l \|_\Sigma 
\lesssim \left| \int_\Sigma u_h^l d\sigma 
-  \int_{\Sigma_h} u_h d \sigma_h \right|
 \lesssim \left| \int_{\Sigma} u_h^l (1 - |\bfB|^{-1}) d\sigma_h \right|
\\
&\qquad
\lesssim h^2 \|u_h^l \|_{\Sigma}
\lesssim h^2 \|\nabla_{\Sigma} u_h^l \|_{\Sigma}
\lesssim h^2 \|\nabla_{\Sigma_h} u_h \|_{\Sigma_h}
\lesssim h^2 \|f\|_{\Sigma_h}
\end{align}

Combining the estimates of $I$ and $II$ we obtain the desired 
estimate.

\end{proof}
\section{Numerical Examples}

\subsection{Condition Number\label{condest2d}}

In order to assess the effect of our stabilization method on the condition number, we discretize a circle, solve the eigenvalue problem of the Laplace-Beltrami operator and sort these in ascending order.
The problem is posed so that the integral of the solution is set to zero by use of a Lagrange multiplier (for well-posedness). The first eigenvalue (corresponding
to the multiplier) is then negative. In the unstabilized method the next eigenvalue is zero with eigenfunction equal to the discrete piecewise 
linear distance function used to define the circle. This is due to our choice to define the circle using a level set function on the same mesh 
used for computations; this particular problem can be avoided using a 
finer mesh for the level set function. We illustrate the effect by showing the zero isoline of the level set function (thick line) together with the isolines of the eigenfunction corresponding to the zero eigenvalue in 
Figure  \ref{zeroeigen}. To avoid this effect, we base the condition 
number on the quotient between the largest eigenvalue and the first 
positive eigenvalue. For the unstabilized method, the first nonzero eigenvalue is thus the third, whereas for the stabilized method it is the second. In the stabilized method we used $\tau_0 = 1/10$ for all computations.

In Fig. \ref{mesh2d} we show the initial position of the circle in a 2D mesh. The circle is then moved to the left, to end up a distance $\delta = 0.1$ to the left of its original position. For each increment $\Delta\delta = 1/100$, we plot the condition numbers of the two methods, shown in Fig. \ref{cond2d}. Note the large variation in condition number of the unstabilized method.

\subsection{Convergence and Conditioning Comparisons}

For our convergence/conditioning comparison, we discretize a sphere of radius $1/2$ 
with center at $(1/2,1/2,1/2)$ and with a load 
\begin{equation}
 f=\frac{6(2x-1)(2y-1)(2z-1)}{3+4x(x-1)+4y(y-1)+4z(z-1)}
\end{equation}
corresponding to the exact solution
\begin{equation}
u=(x-1/2)(y-1/2)(z-1/2)
\end{equation}
compute $f$ on $\Sigma_h$ to solve for $u_h$, and define an approximate $L^2$--error as
\begin{equation}
e_h := \| u_h - {u^e}\|_{L^2(\Sigma_h)}
\end{equation}
A plot of the approximate (unstabilized) solution on a coarse mesh is given in  Figure \ref{solution},
shown on the planes intersected by the level set function.
We compare the error for the stabilized (using different values for $\tau_0$) and unstabilized methods in 
Figure \ref{errors}, where NDOF stands for the total number of degrees
of freedom on the active tetrahedra, so that $h\simeq \text{NDOF}^{-1/2}$.
Note that the error constant is slightly worse for the stabilized methods 
but that all choices converge at the optimal rate of $O(h^2)$. The numbers underlying Figure \ref{errors} are given in Table \ref{table:conv}, where $N$ stands for the number of unknowns, the errors $e_h$ are listed for different $\tau_0$, and $R$ is the rate of convergence.

In Figure \ref{conditioning} we show the condition number computed for the same problem (with the same approach as in Section \ref{condest2d}) with different choices for $\tau_0$ and also for the preconditioning by diagonal scaling suggested in \cite{OlRe10}.
No stabilization results in a condition number that grows faster than the standard rate of $O(h^{-2})$. 
The condition number is most improved by diagonal scaling. Note, however, that diagonal scaling does not remedy the zero eigenvalue induced by the level set.

The numbers underlying Figure \ref{conditioning} are given in Table \ref{table:cond}.

\begin{table}[ht]
\begin{center}
\begin{tabular}{| r | r | r | r | r | r | r |r | r |}
  \hline
N & $\tau_0=1$ & $R$& $\tau_0=0.1$ & $R$ & $\tau_0=0.01$ & $R$ & $\tau_0=0$ & $R$\\
  \hline \hline
  406 & 0.0142 & - & 0.0052 & - & 0.00230 & - & 0.00190 & -  \\ \hline
  1513 & 0.0078 & 0.91   & 0.0017 & 1.70   & 0.00070 & 1.82 & 0.00057 & 1.82   \\ \hline
 6013 & 0.0028  & 1.49  & 0.0004  & 1.93  & 0.00018  & 1.98 & 0.00014 & 2.01  \\ \hline
 24071 & 0.0008  & 1.82  & 0.0001  & 1.97  & 0.00004  & 2.03 & 0.00003 & 2.05  \\ \hline
\end{tabular}
\end{center}
\caption{Errors and convergence for different $\tau_0$\label{table:conv}}
\end{table}

\begin{table}[ht]
\begin{center}
\begin{tabular}{| r | r | r | r | r | r | r |r | r |}
  \hline
N & $\tau_0=1$ & $R$& $\tau_0=0.01$ & $R$ & $\tau_0=0$ & $R$ & Pre & $R$\\
  \hline \hline
  406 & 0.5383 & - & 0.1038 & - & 0.2044 & - & 0.0170 & -  \\ \hline
  1513 & 1.3350 & -1.38   & 0.2001 & -1.00   & 0.4036 & -1.03 & 0.0600 & -1.92   \\ \hline
 6013 & 5.5484  & -2.06  & 0.7595  & -1.93  & 3.5110  & -3.14 & 0.2175 & -1.87  \\ \hline
 24071 & 22.359  & -2.01  & 2.9865  & -1.97  & 69.530  & -4.31 & 0.9354 & -2.10  \\ \hline
\end{tabular}
\end{center}
\caption{Condition numbers $\times 10^{-4}$ and rate for different $\tau_0$ and for diagonal preconditioning.\label{table:cond}}
\end{table}

\section*{Acknowledgements}
{This research was supported in part by EPSRC, UK, Grant No. EP/J002313/1, the Swedish Foundation for Strategic Research Grant No.\ AM13-0029, and the Swedish Research Council Grants Nos.\ 2011-4992 and 2013-4708.}

\newpage
\begin{figure}[h]
\begin{center}
\includegraphics[height=9cm]{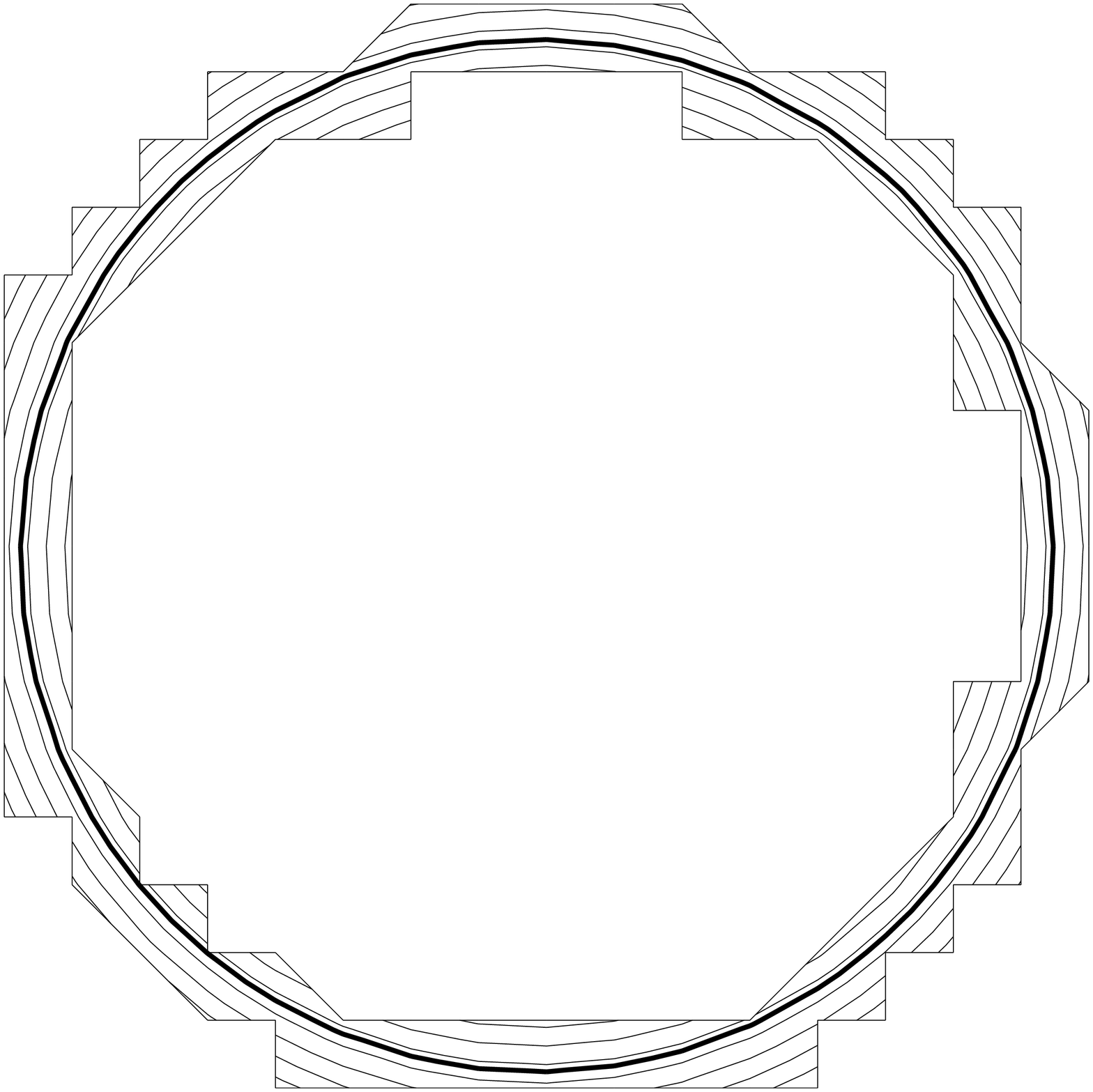}
\end{center}
\caption{Isolines of the eigenfunction corresponding to the zero eigenvalue in the unstabilized method. Level set defining the circle drawn thicker.\label{zeroeigen}}
\end{figure}
\begin{figure}[h]
\begin{center}
\includegraphics[height=9cm]{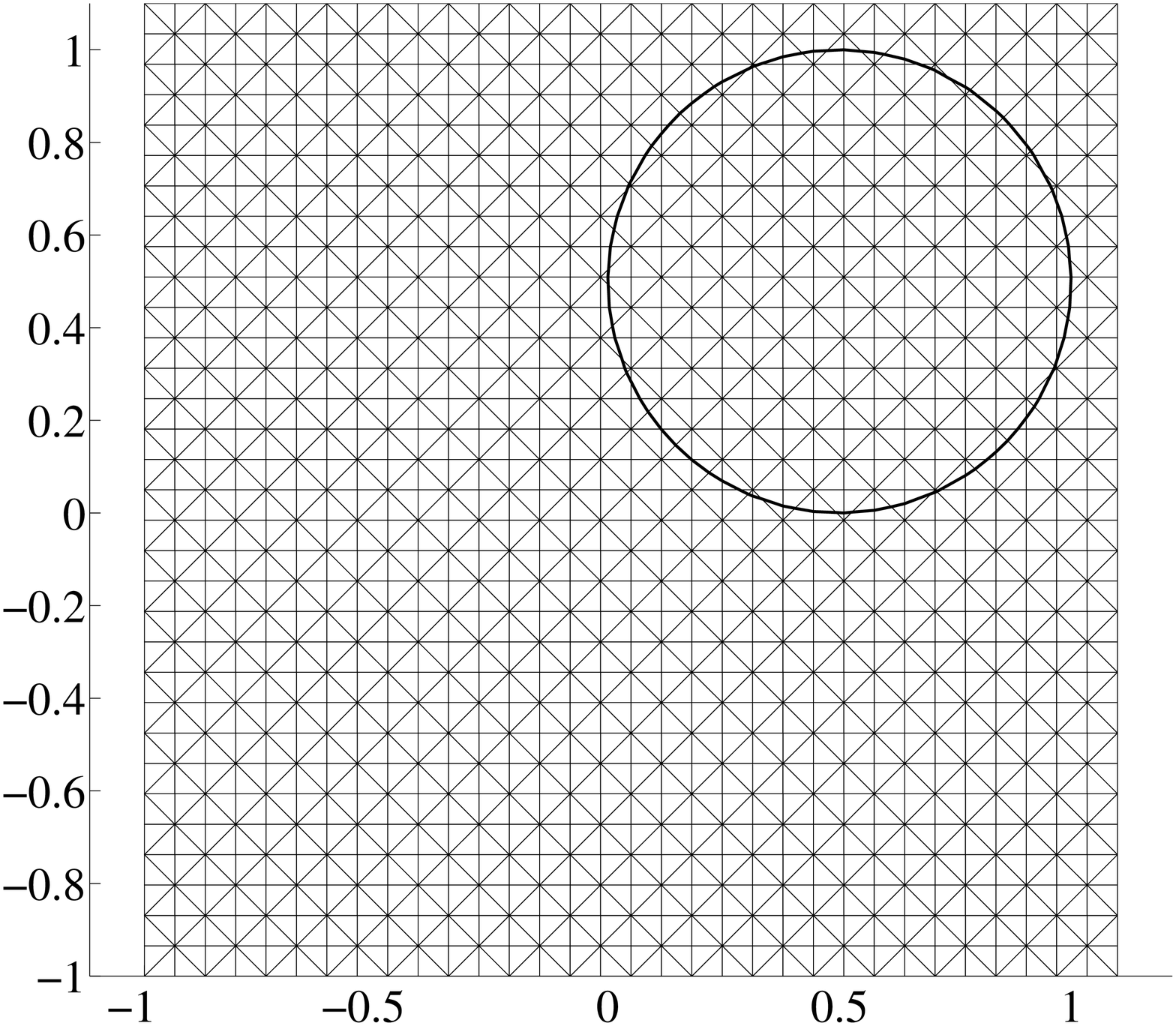}
\end{center}
\caption{Level set isoline used to define the domain in a 2D mesh; initial position.\label{mesh2d}}
\end{figure}
\begin{figure}[h]
\begin{center}
\includegraphics[height=9cm]{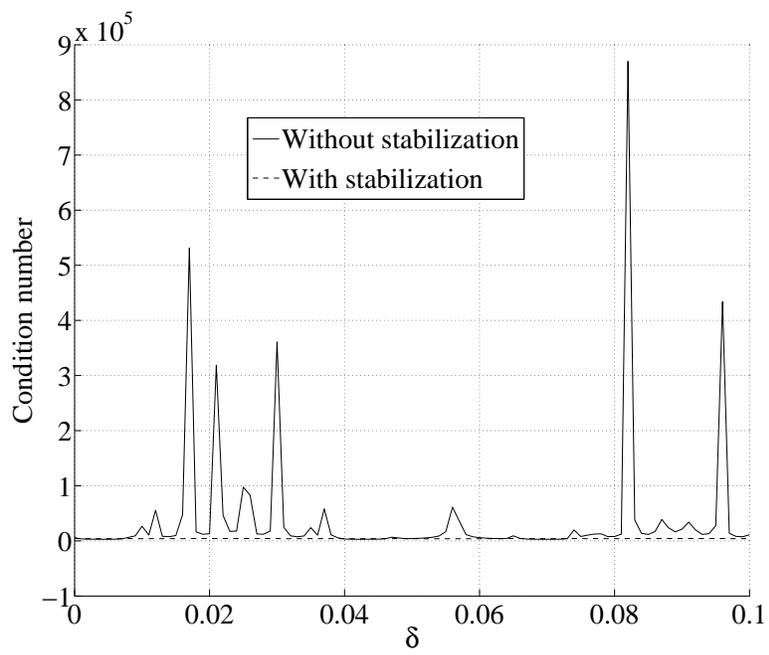}
\end{center}
\caption{Condition numbers for the stabilized and unstabilized methods.\label{cond2d}}
\end{figure}
\begin{figure}[h]
\begin{center}
\includegraphics[height=9cm]{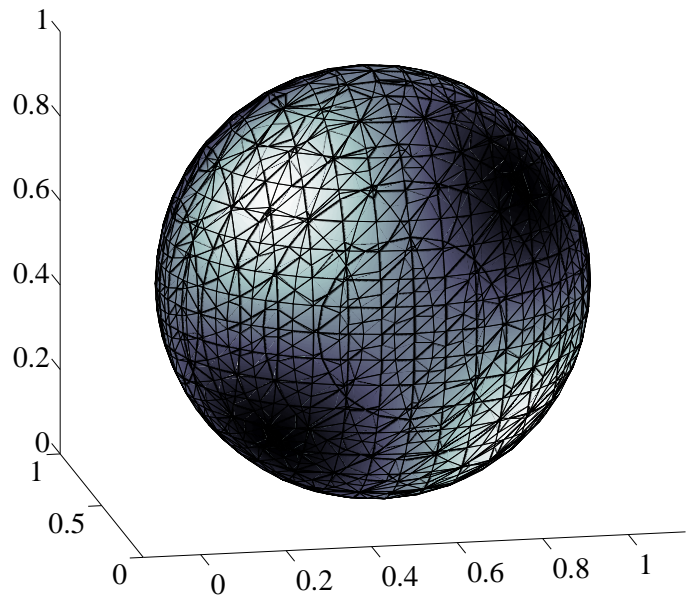}
\end{center}
\caption{Discrete solution on a coarse mesh.\label{solution}}
\end{figure}
\begin{figure}[h]
\begin{center}
\includegraphics[height=9cm]{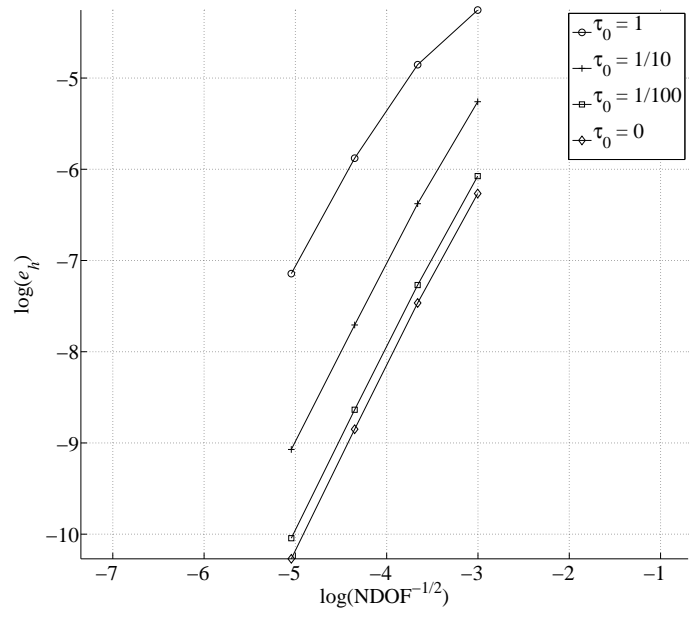}
\end{center}
\caption{Convergence for different choices of $\tau_0$.\label{errors}}
\end{figure}
\begin{figure}[h]
\begin{center}
\includegraphics[height=9cm]{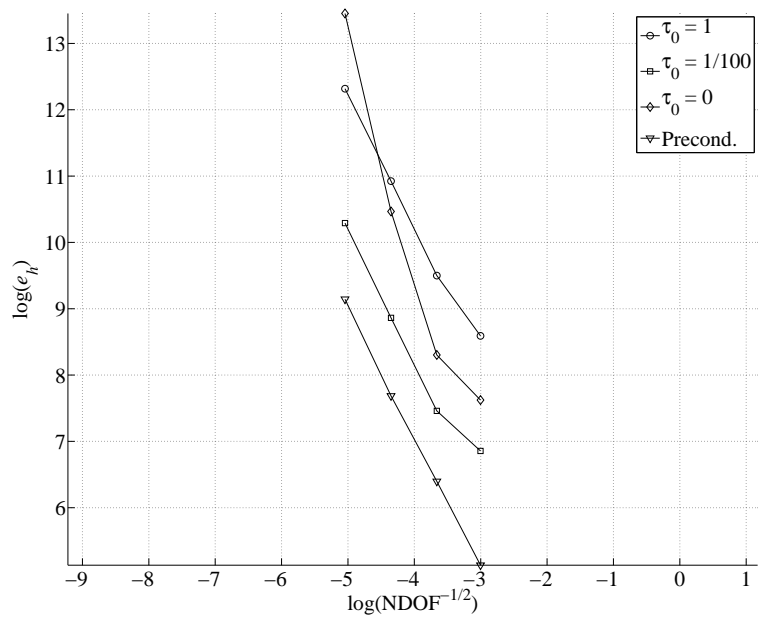}
\end{center}
\caption{Condition numbers for different choices of $\tau_0$ and for preconditioning by diagonal scaling.\label{conditioning}}
\end{figure}

\end{document}